\documentclass[a4paper,12pt]{amsart}

\usepackage[margin=3truecm,top=4truecm,bottom=4truecm]{geometry}
\usepackage{amsthm,amsmath,amssymb,setspace,mleftright,constants,cite,mathtools}
\usepackage{graphicx}
\usepackage[unicode=true,hidelinks=true]{hyperref} 

\allowdisplaybreaks 
\usepackage[pagewise]{lineno}
\usepackage{etoolbox} 
\makeatletter 
\newcommand*\linenomathpatch{\@ifstar{\linenomathpatch@AMS}{\linenomathpatch@}}
\newcommand*\linenomathpatch@[1]{
	\expandafter\pretocmd\csname #1\endcsname {\linenomathWithnumbers}{}{}
	\expandafter\pretocmd\csname #1*\endcsname{\linenomathWithnumbers}{}{}
	\expandafter\apptocmd\csname end#1\endcsname {\endlinenomath}{}{}
	\expandafter\apptocmd\csname end#1*\endcsname{\endlinenomath}{}{}
}
\newcommand*\linenomathpatch@AMS[1]{
	\expandafter\pretocmd\csname #1\endcsname {\linenomathWithnumbersAMS}{}{}
	\expandafter\pretocmd\csname #1*\endcsname{\linenomathWithnumbersAMS}{}{}
	\expandafter\apptocmd\csname end#1\endcsname {\endlinenomath}{}{}
	\expandafter\apptocmd\csname end#1*\endcsname{\endlinenomath}{}{}
}
\let\linenomathWithnumbersAMS\linenomathWithnumbers
\patchcmd\linenomathWithnumbersAMS{\advance\postdisplaypenalty\linenopenalty}{}{}{}
\makeatother 

\linenomathpatch{equation}
\linenomathpatch*{gather}
\linenomathpatch*{multline}
\linenomathpatch*{align}
\linenomathpatch*{alignat}
\linenomathpatch*{flalign}

\makeatletter

\@addtoreset{equation}{section}
\makeatother

\newconstantfamily{E}{symbol=\mathcal{E}}
\newconstantfamily{G}{symbol=\mathcal{G}}

\theoremstyle{plain} 
\newtheorem{thm}{Theorem}[section]
\newtheorem{prop}[thm]{Proposition}
\newtheorem{cor}[thm]{Corollary}
\newtheorem{lem}[thm]{Lemma}

\theoremstyle{definition}

\theoremstyle{remark}
\newtheorem{rem}[thm]{Remark}

\newcommand{\N}{\mathbb{N}}
\newcommand{\R}{\mathbb{R}}
\newcommand{\Z}{\mathbb{Z}}

\renewcommand{\P}{\mathbb{P}}
\newcommand{\E}{\mathbb{E}}

\newcommand{\1}[1]{\mathbf{1}_{#1}}

\newcommand{\half}{\frac{1}{2}}

\renewcommand{\tilde}{\widetilde}

\newcommand{\hyphen}{\textrm{-}}
\newcommand{\as}{\textrm{a.s.}}

\newcommand{\bP}{\mathbf{P}}
\newcommand{\bE}{\mathbf{E}}

\begin{document}
\title[Rate of divergence of time constant for frog model]{Rate of divergence of time constant for frog model with vanishing initial density}

\author[R.~FUKUSHIMA]{Ryoki FUKUSHIMA}
\address[R.~FUKUSHIMA]{Department of Mathematics, University of Tsukuba, Ibaraki 305-8571, Japan.}
\email{ryoki@math.tsukuba.ac.jp}

\author[N.~KUBOTA]{Naoki KUBOTA}
\address[N.~KUBOTA]{College of Science and Technology, Nihon University, Chiba 274-8501, Japan.}
\email{kubota.naoki08@nihon-u.ac.jp}


\begin{abstract}
The frog model with a Bernoulli initial configuration is an interacting particle system on the $d$-dimensional lattice ($d \geq 2$) with two types of particles: active and sleeping.
Active particles perform independent simple random walks.
In contrast, although the sleeping particles do not move at first, they become active and start moving once touched by the active particles.
Initially, only the origin has a single active particle, and the other sites have sleeping particles according to a Bernoulli distribution.
After the original active particle starts moving, further active particles are gradually generated under the above rule and propagate across the lattice.
The time required for the propagation of active frogs is expected to increase as the parameter of the Bernoulli distribution decreases, since fewer frogs are available.
The aim of this paper is to investigate this increase in the vanishing density limit.
In particular, we observe that it diverges and the rate of divergence differs significantly between $d=2$ and $d \geq 3$.
\end{abstract}

\maketitle

\section{Introduction and main results}
In this paper, we consider the frog model with Bernoulli initial configuration on the $d$-dimensional lattice $\Z^d$ ($d \geq 2$).
This model is an interactive particle system on $\Z^d$ that consists of two types of particles: active and sleeping.
Active particles perform independent simple random walks on $\Z^d$.
On the other hand, although sleeping particles do not move at first, they become active and start moving when touched by active particles.
Initially, only the origin $0$ of $\Z^d$ has one active particle, while the other sites have sleeping particles according to a Bernoulli distribution.
After the original active particle starts moving, the number of active particles gradually increases and they propagate across $\Z^d$.
Our main object of interest is the so-called time constant, which describes the speed of propagation of active particles.
The time constant is a function of the parameter of the Bernoulli distribution, and the aim of this paper is to investigate the rate of divergence of the time constant as the parameter goes to zero.
More precisely, letting $r \in (0,1]$ be the parameter of the Bernoulli distribution, we prove that as $r \searrow 0$, the time constant behaves like $\sqrt{|\log r|/r}$ if $d=2$ and $1/\sqrt{r}$ if $d \geq 3$.

\subsection{The model}\label{subsect:model}
Let $d \geq 2$ and denote by $\mathcal{P}$ the set of all probability measures on $\N_0:=\N \cup \{0\}$ not concentrated in zero.
For a given $\Phi \in \mathcal{P}$, we consider a family $\omega=(\omega(x))_{x \in \Z^d}$ of independent random variables with common law $\Phi$.
Furthermore, independently of $\omega$, let $S=((S_k^x(\ell))_{k=0}^\infty)_{x \in \Z^d,\ell \in \N}$ be a family of independent simple random walks on $\Z^d$ satisfying that $S_0^x(\ell)=x$ for each $x \in \Z^d$ and $\ell \in \N$.
Then, for any $x,y \in \Z^d$, the \emph{first passage time} $T(x,y)=T(x,y,\omega,S)$ from $x$ to $y$ is defined as
\begin{align}\label{def:PassageTime}
	T(x,y):=\inf\mleft\{ \sum_{i=0}^{m-1} \tau(x_i,x_{i+1}):
	\begin{minipage}{11.7em}
		$m \in \N$, $x_0,x_1,\dots,x_m \in \Z^d$\\
		with $x_0=x$, $x_m=y$
	\end{minipage}
	\mright\},
\end{align}
where
\begin{align*}
	\tau(x_i,x_{i+1})
	&= \tau\bigl( x_i,x_{i+1},\omega(x_i),S_\cdot^{x_i}(\cdot) \bigr)\\
	&:= \inf\bigl\{ k \geq 0:S_k^{x_i}(\ell)=x_{i+1} \text{ for some } \ell \in [1,\omega(x_i)] \bigr\},
\end{align*}
with the convention that $\tau(x_i,x_{i+1}):=\infty$ if $\omega(x_i)=0$.
Note that the first passage time satisfies the triangle inequality:
\begin{align}\label{eq:triangle}
	T(x,z) \leq T(x,y)+T(y,z),\qquad  x,y,z \in \Z^d.
\end{align}

On the event $\{ \omega(0) \geq 1 \}$, the first passage time $T(0,y)$ is intuitively interpreted as follows:
First, we place ``frogs'' on $\Z^d$ according to the initial configuration $\omega$, i.e., $\omega(x)$ frogs sit on each site $x$ (there is no frog at $x$ if $\omega(x)=0$).
In particular, the event $\{ \omega(0) \geq 1 \}$ ensures that at least one frog is assigned to the origin $0$.
The behavior of the $\ell$-th frog sitting on site $x$ is governed by the simple random walk $S_\cdot^x(\ell)$, but not all frogs move around from the beginning.
Initially, the only frogs sitting on $0$ are active and perform independent simple random walks, while the other frogs are sleeping and do not move.
Sleeping frogs become active and start performing simple random walks independently once they are touched by original active frogs.
When we repeat this procedure for the generated active and remaining sleeping frogs, $T(0,y)$ represents the minimum time at which an active frog reaches $y$.

Alves et al.~\cite[Section~2]{AlvMacPopRav01} proved that for a given $\Phi \in \mathcal{P}$, there exists a nonrandom norm $\mu(\cdot)=\mu_\Phi(\cdot)$ on the $d$-dimensional Euclidean space $\R^d$, which is called the \emph{time constant}, such that almost surely on the event $\{ \omega(0) \geq 1 \}$,
\begin{align}\label{eq:tc}
	\lim_{n \to \infty} \frac{1}{n}T(0,nx)=\mu(x),\qquad x \in \Z^d.
\end{align}
Furthermore, $\mu(\cdot)$ is invariant under lattice symmetries and satisfies
\begin{align}
\label{eq:tc_bound}
	\|x\|_1 \leq \mu(x) \leq \mu(\xi_1)\|x\|_1,\qquad x \in \R^d,
\end{align}
where $\|\cdot\|_1$ and $\xi_1$ are the $\ell^1$-norm and the first coordinate vector on $\R^d$, respectively.
As a consequence of \eqref{eq:tc}, Alves et al.~\cite[Theorem~{1.1}]{AlvMacPopRav01} also proved the so-called \emph{shape theorem}.
This gives the asymptotic behavior of the random set $\mathcal{B}(t):=\{x \in \Z^d:T(0,x) \leq t\}$ (which is the set of all sites visited by active frogs by time $t$) as $t \to \infty$:
given $\epsilon>0$, almost surely on the event $\{ \omega(0) \geq 1 \}$, for all large $t$,
\begin{align}\label{eq:shape}
	(1-\epsilon)t\mathcal{B} \cap \Z^d \subset \mathcal{B}(t) \subset (1+\epsilon)t\mathcal{B} \cap \Z^d,
\end{align}
where $\mathcal{B}=\mathcal{B}_\Phi:=\{ x \in \R^d:\mu_\Phi(x) \leq 1 \}$, the \emph{asymptotic shape}.
It is clear from the properties of the time constant that the asymptotic shape $\mathcal{B}$ is a convex, compact set with nonempty interior.

\subsection{Main results}\label{subsect:main}
Let $0<r \leq 1$ and denote by $\text{Ber}(r)$ the Bernoulli distribution with parameter $r$.
We consider the initial configuration $\omega$ governed by the product of $\text{Ber}(r)$ and write $\P_r$ for the underlying probability measure of $\omega$:
\begin{align*}
	\P_r(\omega(x)=1)=1-\P_r(\omega(x)=0)=r,\qquad x \in \Z^d.
\end{align*}
A site $x$ is said to be \emph{occupied} if $\omega(x)=1$, and $\mathcal{O}=\mathcal{O}(\omega)$ stands for the set of all occupied sites, i.e.,
\begin{align}\label{eq:occupied}
	\mathcal{O}:=\{ x \in \Z^d:\omega(x)=1 \}.
\end{align}
Moreover, in the present setting, since each site initially has at most one frog, the label ``$1$'' is usually omitted from the random walk notation as follows:
\begin{align*}
	S=(S_\cdot^x(1))_{x \in \Z^d}=(S_\cdot^x)_{x \in \Z^d}.
\end{align*}
Moreover, denote by $P$ the underlying probability measure of the family of independent simple random walks $S=(S_\cdot^x)_{x \in \Z^d}$.
Then, by setting $\bP_r:=\P_r \times P$, the time constant $\mu_r(\cdot):=\mu_{\text{Ber}(r)}(\cdot)$ in the Bernoulli setting can be described as follows:
$\bP_r \hyphen \as$ on the event $\{ \omega(0)=1 \}$,
\begin{align*}
	\lim_{n \to \infty} \frac{1}{n}T(0,nx)=\mu_r(x),\qquad x \in \Z^d.
\end{align*}

Define for any $d \geq 2$ and $r \in (0,1]$,
\begin{align}\label{eq:delta}
	\delta_d(r):= 
	\begin{dcases*}
	    \sqrt{\dfrac{|\log r|}{r}}, & if $d=2$,\\[0.5em]
	    \frac{1}{\sqrt{r}}, & if $d \geq 3$.
	\end{dcases*}
\end{align}
The following theorems are the main results of this article.

\begin{thm}\label{thm:divergence}
There exists a constant $A_1>0$ (which depends only on $d$) such that if $r \in (0,1]$ is small enough (depending on $d$), then for all $x \in \R^d$,
\begin{align*}
	\mu_r(x) \geq A_1\delta_d(r)\|x\|_1.
\end{align*}
In particular, we have for any $x \in \R^d \setminus \{0\}$,
\begin{align*}
	\lim_{r \searrow 0} \mu_r(x)=\infty.
\end{align*}
\end{thm}

\begin{thm}\label{thm:upper}
There exists a constant $A_2>0$ (which depends only on $d$) such that if $r \in (0,1]$ is small enough (depending on $d$), then for all $x \in \R^d$,
\begin{align*}
	\mu_r(x) \leq A_2\delta_d(r)\|x\|_1.
\end{align*}
\end{thm}

The above theorems tell us that for each $x \in \R^d \setminus \{0\}$, $\mu_r(x)$ behaves like $\delta_d(r)\|x\|_1$ as $r \searrow 0$.
As a corollary, we can explicitly determine the rate of decline of the asymptotic shape $\mathcal{B}_r:=\mathcal{B}_{\text{Ber}(r)}$ as $r \searrow 0$.

\begin{cor}
Let $A_1$ and $A_2$ be the constants appearing in Theorems~\ref{thm:divergence} and \ref{thm:upper}, respectively.
If $r \in (0,1]$ is small enough (depending on $d$), then
\begin{align*}
	B_1\bigl( 0,A_2^{-1}\delta_d(r) \bigr) \subset \mathcal{B}_r \subset B_1\bigl( 0,A_1^{-1}\delta_d(r) \bigr),
\end{align*}
where $B_1(0,R):=\{ y \in \R^d:\|y\|_1 \leq R \}$ stands for the $\ell^1$-ball of center $0$ and radius $R>0$.
\end{cor}

\subsection{Related works}
\label{sec:literature}
Let us comment on earlier works related to the above results.
Although the present article focuses on the Bernoulli initial configuration, the frog model on $\Z^d$ has also been investigated for general initial configurations.
Alves et al.~\cite{AlvMacPop02} first proved the shape theorem \eqref{eq:shape} in the one-frog-per-site setting (i.e., $\omega(0)$ is distributed as $\text{Ber}(1)$).
After that, in \cite{AlvMacPopRav01}, they introduced the frog model with random initial configuration and observed that the shape theorem also holds for more general initial configurations.
As mentioned in Section~\ref{subsect:model}, the time constant plays the key role in establishing the shape theorem, and hence most of \cite{AlvMacPopRav01} is devoted to the study of the behavior of the first passage time.
In particular, the main contribution of \cite{AlvMacPopRav01} is to prove the integrability of the first passage time.
This integrability, together with the triangle inequality \eqref{eq:triangle}, enables us to apply the subadditive ergodic theorem, and the time constant is constructed as the asymptotic behavior of the first passage time (see \eqref{eq:tc}). 
However, the subadditive ergodic theorem only guarantees the existence of the time constant, and provides little information on its properties.
A nontrivial property revealed in \cite[Theorem~{1.2}]{AlvMacPopRav01} is that if the law $\Phi \in \mathcal{P}$ satisfies $\Phi([t,\infty)) \geq (\log t)^{-\delta}$ for some $\delta \in (0,d)$ and for all large $t>0$, then the time constant $\mu_\Phi(\cdot)$ for $\Phi$ coincides with the $\ell^1$-norm on $\R^d$.
This suggests that the law of the initial configuration may influence the time constant (or equivalently the propagation of active frogs).
There are several articles that examine the effect of the law of initial configuration on the time constant.
Johnson and Junge~\cite[Corollary~9]{JohJun18} introduced a (nontrivial) stochastic order on the initial configuration, and proved that the time constant is non-increasing with respect to this order.
Furthermore, the second author~\cite{Kub20} showed the continuity for the time constant in the law of initial configuration.
These two works raise the question of whether the time constant varies strictly when the law of the initial configuration is changed.
Concerning the Bernoulli initial configuration, which is also considered in the present article, Can et al.~\cite{CanKubNak25} provided the following positive answer to this question:
Let $0<r_0<1$.
Then, there exist constants $C,C'>0$ (which depend only on $d$ and $r_0$) such that if $r_0 \leq p<q \leq 1$, then for all $x \in \R^d \setminus \{0\}$,
\begin{align}\label{eq:Lipschitz}
	C(q-p) \leq \frac{\mu_p(x)-\mu_q(x)}{\|x\|_1} \leq C'(q-p).
\end{align}
In particular, \eqref{eq:Lipschitz} implies that $\mu_r(x)$ is locally Lipschitz continuous in $r \in (0,1]$ (i.e., $\mu_r(x)$ is Lipschitz continuous on every closed interval contained in $(0,1]$).
Moreover, since the parameter $r_0$ can be chosen arbitrarily small, $\mu_r(x)$ is strictly decreasing in $r \in (0,1]$.
We emphasize here that the constants $C$ and $C'$ above depend on the parameter $r_0$, and hence \eqref{eq:Lipschitz} does not describe the behavior of $\mu_r(x)$ as $r \searrow 0$.
Solving this problem is a motivation of the present article, and Theorems~\ref{thm:divergence} and \ref{thm:upper} actually provide the optimal order of $\mu_r(x)$ as $r \searrow 0$.

The aforementioned articles focus on the time constant of the frog model on $\Z^d$ itself.
However, in view of the fact that almost surely, $T(0,x) \approx \mu(x)$ when $\|x\|_1$ is large enough (see \eqref{eq:tc} and \eqref{eq:shape}), it is also important to analyze the difference between the first passage time and the time constant.
Representative quantities in this context are the conditional probabilities of the following unlikely events, given that $\omega(0) \geq 1$ occurs: $\{ T(0,x) \geq (1+\epsilon)\mu(x) \}$ and $\{ T(0,x) \leq (1-\epsilon)\mu(x) \}$ for $\epsilon>0$ and $x \in \Z^d$.
The conditional probability of the former (resp.~latter) event is called the upper (resp.~lower) tail large deviation probability for the first passage time, and the second author~\cite{Kub19} proved that the upper and lower large deviation probabilities exhibit at least stretched-exponential decay.
For the one-frog-per-site setting, \cite{CanKubNak25_IP} and \cite{CanKubNak23_arXiv} estimated the upper tail large deviation probability precisely and identified its optimal decay rate.
The technique developed there may also be useful for analyzing the number of active frogs generated over time for more general initial configurations.
In fact, to guarantee that many active frogs activate many sleeping frogs, we used a result from \cite{CanKubNak25_IP} in the present article (see Proposition~\ref{prop:CKN} below).
On the other hand, there are few results regarding the fluctuation of the first passage time.
The main reason is that the order of the variance of the first passage is still unidentified.
To our best knowledge, the work of Can and Nakajima~\cite{CanNak19} is the latest result in this direction, where it is proved that in the one-frog-per-site setting, the variance of $T(0,x)$ grows at most sublinearly in $\|x\|_1$.

So far, we have considered the frog model driven by discrete-time simple random walks on $\Z^d$, but the frog model admits several variants.
Ram\'{\i}rez et al.~\cite{RamSid04} treated the continuous-time frog model on $\Z^d$, in which frogs perform continuous-time simple random walks on $\Z^d$, and proved the corresponding shape theorem.
In the one-dimensional case, B\'{e}rard and Ram\'{\i}rez \cite{BerRam10,BerRam16} investigated the central limit theorem and the large deviation problem mentioned above for the position of the rightmost site visited by active frogs (note that in the one-dimensional case, the position of the rightmost visited site and the first passage time of a site at the right of the origin are interchangeable).
On the other hand, Beckman et al.~\cite{BecDinDurHuoJun18} introduced the Brownian frog model (i.e., sleeping frogs are placed according to a Poisson point process and active frogs perform independent Brownian motions on $\R^d$), and they analyzed the long-time behavior of the propagation of active frogs, including the shape theorem.

Apart from the behavior of the first passage time, the recurrence/transience problem has also been actively investigated in the frog model.
In \cite[Section~{2.4}]{TelWor99} (which is the first published paper on the frog model), Telcs and Wormald proved that for all $d \geq 1$, the discrete-time frog model on $\Z^d$ with the one-frog-per-site configuration is recurrent, i.e., with probability one, active frogs visit $0$ infinitely often (the frog model is said to be transient if it is not recurrent).
Subsequently, this recurrence/transience problem has been investigated in various settings: the structure of the underlying graph ($\Z^d$, trees and so on), the law of the initial configuration and the presence/absence of the drift in each frog.
For further details, we refer the reader to \cite{JohRol19,HofJohJun17,MulWie20,MicRos20,KosZer17,GuoTanWei22} and the references therein.

\subsection{Questions on the scaling limit}\label{subsect:scaling}

The main results of this paper determine the asymptotics of the time constant in the limit $r\searrow 0$ up to a multiplicative constant. The natural next questions are the existence and characterization of the scaling limits. 

If we naively take limits of the initial configuration and the random walks, then they converge to the Poisson point process and the Brownian motions, respectively. However, for these limiting objects, the dynamics cannot be defined, as the Brownian motions do not hit points. It might be better to take an average over the configuration of initial particles so that the limiting process would be a branching Brownian motions with space-time dependent branching rate. More modest but still interesting goal is to show that $\delta_d(r)^{-1}\mathcal{B}_r$ converges to the Euclidean ball as $r\searrow 0$. 
We leave these questions for future research. 

\subsection{Organization of the paper}\label{subsect:org}
Let us describe how the present article is organized.
In Section~\ref{sect:lowerbound}, we prove Theorem~\ref{thm:divergence}, which provides the lower bound for the time constant.
The basic idea for the proof comes from \cite[Sections~3 and 4]{BecDinDurHuoJun18}, and its key object is a chain of active frogs, which is constructed, roughly speaking, by splicing together the trajectories of several active frogs at some points (see below \eqref{eq:range} for more details).
By construction, there exists a chain of active frogs achieving the first passage time $T(0,nx)$ (see Lemma~\ref{lem:T_chain}), and hence the task is to show that there is no chain of active frogs that starts from $0$ and reaches $nx$ within time $\delta_d(r)\|x\|_1$ up to a constant factor.
Once this claim is established, the desired lower bound of Theorem~\ref{thm:upper} follows from \eqref{eq:tc}.

Section~\ref{sect:upper} is devoted to the proof of Theorem~\ref{thm:upper}, which provides the upper bound for the time constant.
The rough strategy of the proof is as follows:
Tessellate $\Z^d$ with boxes whose side length is of order $r^{-d/2}$.
We show in Proposition~\ref{prop:good} that if an active frog exists nearby the center of such a box, then with high probability, it can generate sufficiently many active frogs in that box and one of them reaches nearby the centers of the neighboring boxes, within time $r^{-d/2}\delta_d(r)$ (such a box is called $r$-good, see at the beginning of Section~\ref{subsect:pf_upper} for more details).
Once this claim is established, we can apply a percolation argument to find a cluster of $r$-good boxes that bridges two sites $0$ and $nx$.
Since the side length of each $r$-good box is of order $r^{-d/2}$, it follows that for some constant $A>0$,
\begin{align*}
	T(0,nx) \leq Ar^{-d/2}\delta_d(r) \times (n\|x\|_1/r^{-d/2})=An\delta_d(r)\|x\|_1,
\end{align*}
and the desired upper bound of Theorem~\ref{thm:upper} is obtained due to \eqref{eq:tc}.

Thus the crux of the proof of the upper bound is Proposition~\ref{prop:good}. Its proof is given in Section~\ref{subsect:good} and is divided into three parts:

\vspace{0.5em}

\paragraph*{\underline{Step~1 (Section~\ref{subsect:sowing})}}
Given a sufficiently small $\epsilon>0$ and any box whose side length is of order $r^{1/2+\epsilon}$, we sow ``seeds'' of active frogs in and around that box within time $r^{-(1+3\epsilon)}$.
Proposition~\ref{prop:sowing} guarantees the possibility of this sowing.

\vspace{0.5em}

\paragraph*{\underline{Step~2 (Section~\ref{subsect:activating})}}
Using the seeds of active frogs sown in Step~1, we show that with high probability, all sites in the box, centered at the same point as in Step~1 and with side length $r^{-(d+1)/4}$, are visited by active frogs within time $5dr^{-(d+1)/2}$ (see Proposition~\ref{prop:activating}).

\vspace{0.5em}

\paragraph*{\underline{Step~3 (Section~\ref{subsect:recursion})}}
Due to Step~2, if an active frog exists nearby any initial box of side length $r^{-(d+1)/4}$, then all sleeping frogs in that box can be activated within time $5dr^{-(d+1)/2}$.
However, it is not enough to prove Proposition~\ref{prop:good}, since $r^{-(d+1)/4} \ll r^{-d/2}$ and active frogs must reach a region whose distance from the center of the initial box above is of order $r^{-d/2}$.
We use another recursion to generate active frogs starting from the initial box and propagate one of them to the target region, within the required time $r^{-d/2}\delta_d(r)$. To make sure that this happens with high probability, we use Lemma~\ref{lem:CKN_cor} which is a consequence of a result obtained in \cite{CanKubNak25_IP} (see Proposition~\ref{prop:CKN} below) combined with a simple concentration bound. We reproduce the proof of Proposition~\ref{prop:CKN} in Appendix~\ref{app:CKN} for the reader's convenience.

\vspace{0.5em}


We close this section with some general notation.
Denote by $\E_r$, $E$ and $\bE_r$ the expectations associated to the probability measures $\P_r$, $P$ and $\bP_r$ stated in Section~\ref{subsect:main}, respectively.
Moreover, for each $i \in \{ 1,2,\infty \}$, the $\ell^i$-norm on $\R^d$ is designated by $\|\cdot\|_i$ and let $B_i(x,R):=\{ y \in \R^d:\|y-x\|_i \leq R \}$ be the $\ell^i$-ball in $\R^d$ of center $x \in \R^d$ and radius $R>0$.
Finally, throughout the paper, we use $c$ and $c'$ to denote arbitrary positive constants which may change from line to line but \emph{depend only on the dimension $d$}.

\section{Lower bound for the time constant}\label{sect:lowerbound}
In this section, we prove Theorem~\ref{thm:divergence}, which says that for all $r \in (0,1]$ sufficiently small, the time constant $\mu_r(x)$ is bounded from below by $\delta_d(r)\|x\|_1$, up to a multiplicative constant.
The basic idea for the proof is essentially the same as in~\cite[Sections~3 and 4]{BecDinDurHuoJun18}, and the task is to control chains of active frogs, which are constructed, roughly speaking, by splicing together the trajectories of several active frogs at some points.
As we state in Lemma~\ref{lem:T_chain} below, the first passage time $T(0,nx)$ is regarded as the minimum duration of chains of active frogs that start at $0$ and end at $nx$.
Hence, due to \eqref{eq:tc}, it suffices for the proof of Theorem~\ref{thm:divergence} to show that there is no chain of active frogs that starts from $0$ and reaches $nx$ within time $\delta_d(r)n\|x\|_1$ up to a constant factor.
To do this, we first derive an exponential probability bound for each chain (see Lemma~\ref{lem:SingleChainTail} below) and then use the union bound over all chains.

For the reader's convenience, we begin by formally describing a chain of active frogs.
Given a (deterministic or random) finite sequence $I=(I_\ell)_{\ell=1}^\nu$ in $\N$, the trajectories of active frogs are spliced together as follows:
The active frog forcibly assigned to $0$ is tracked until it has visited $I_1$ occupied sites, and then it is removed.
We activate only the frog on the most recently visited site and begin tracking it anew.
On the other hand, any frogs on the other visited sites are removed.
By repeating this procedure recursively for each active frogs, the chain of active frogs induced by $I$ is constructed.
That is, given the $\ell$-th active frog to be tracked (where $\ell \in [2,\nu)$), this active frog is tracked until it has visited $I_{\ell+1}$ occupied sites, after which it is removed.
Subsequently, we activate only the frog on the most recently visited site, begin tracking it anew, and remove any frogs on the other visited sites.

To make the above construction precise, let us first define for any $A \subset \Z^d$ and $n \in \N_0$,
\begin{align}\label{eq:range}
	\mathcal{R}_n^A:=\bigcup_{x \in A} \{ S_k^x:0 \leq k \leq n \},
\end{align}
which is the set of all sites visited by the simple random walks starting from $A$ on or before time $n$.
In particular, write $\mathcal{R}_n^x:=\mathcal{R}_n^{\{ x \}}$ for simplicity.
Next, suppose that an arbitrary finite sequence $I=(I_\ell)_{\ell=1}^\nu$ in $\N$ is given.
Setting $a_I(1):=0$ and $\sigma_I(1,0):=0$, we define for each integer $i \in [1,I_1]$,
\begin{align}
\label{eq:sigma1}
	\sigma_I(1,i):=\inf\Bigl\{ k>\sigma_I(1,i-1):S^{a_I(1)}_k \in \mathcal{O} \setminus \mathcal{R}^{a_I(1)}_{\sigma_I(1,i-1)} \Bigr\},
\end{align}
where $\mathcal{O}$ is the set of all occupied sites (see \eqref{eq:occupied}).
Tracking of $S_\cdot^{a_I(1)}$ stops at step $\sigma_I(1,I_1)$, which is the time when $S_\cdot^{a_I(1)}$ visits $I_1$ distinct occupied sites, excluding $0$.
If $\nu=1$, then we complete the construction and the chain of active frogs induced by $I$ is exactly the trajectory of $S_\cdot^0$ up to $\sigma_I(1,I_1)$ steps.
If $\nu \geq 2$, then we continue to track simple random walks recursively as follows:
Given $a_I(j)$, $\sigma_I(j,I_j)$, $1 \leq j \leq \ell$ (where $\ell \in [1,\nu)$), let
\begin{align*}
	a_I(\ell+1):=S_{\sigma_I(\ell,I_\ell)}^{a_I(\ell)},\quad \sigma_I(\ell+1,0):=0,
\end{align*}
and define for any integer $i \in [1,I_{\ell+1}]$, 
\begin{align*}
	&\sigma_I(\ell+1,i)\\
	&:=\inf\Biggl\{ k>\sigma_I(\ell+1,i-1):S^{a_I(\ell+1)}_k \in \mathcal{O} \setminus \Biggl( \bigcup_{j=1}^\ell \mathcal{R}^{a_I(j)}_{\sigma_I(j,I_j)} \cup \mathcal{R}^{a_I(\ell+1)}_{\sigma_I(\ell+1,i-1)} \Biggr) \Biggr\}.
\end{align*}

Our first observation is that there exists a chain of active frogs realizing the first passage time from $0$.

\begin{lem}\label{lem:T_chain}
Assume that $y \in \mathcal{O}$ and $T(0,y)<\infty$.
Then, there exists a (random) finite sequence $I=(I_\ell)_{\ell=1}^\nu$ in $\N$ such that
\begin{align*}
	T(0,y)=\sum_{\ell=1}^{\nu} \sigma_I(\ell,I_\ell).
\end{align*}
\end{lem}
\begin{proof}
Assume that $y \in \mathcal{O}$ and $T(0,y)<\infty$.
Due to \eqref{def:PassageTime}, there exist $\nu \in \N$ and distinct $x_0,x_1,\dots,x_\nu \in \mathcal{O}$ with $x_0=0$ and $x_\nu=y$ such that
\begin{align*}
	T(0,y)=\sum_{\ell=1}^\nu \tau(x_{\ell-1},x_\ell).
\end{align*}
We now choose a finite sequence $I=(I_\ell)_{\ell=1}^\nu$ in $\N$ as follows:
Denote by $I_1$ the number of distinct occupied sites (excluding $0$) that $S_\cdot^0$ visits up to and including its arrival at $x_1$.
Then, for any $\ell \in [2,\nu]$, we recursively define $I_\ell$ as the number of distinct occupied sites, excluding those in $\bigcup_{j=1}^{\ell-1} \mathcal{R}^{x_{j-1}}_{\tau(x_{j-1},x_j)}$, that $S_\cdot^{x_{\ell-1}}$ visits up to and including its arrival at $x_\ell$.
Since $x_0,x_1,\dots,x_\nu \in \mathcal{O}$ are distinct, it is easy to see that $\tau(x_{\ell-1},x_\ell)=\sigma_I(\ell,I_\ell)$ holds for each $\ell \in [1,\nu]$.
This implies that
\begin{align*}
	T(0,y)=\sum_{\ell=1}^\nu \tau(x_{\ell-1},x_\ell)=\sum_{\ell=1}^\nu \sigma_I(\ell,I_\ell),
\end{align*}
and hence we can construct the desired finite sequence $I=(I_\ell)_{\ell=1}^\nu$ in $\N$.
\end{proof}

From Lemma~\ref{lem:T_chain}, the key to proving Theorem~\ref{thm:divergence} is to control the behavior of chains of active frogs.
To do this, we prepare the following lemma, which provides some estimates for the duration and the range of the chain of active frogs induced by any deterministic finite sequence in $\N$.

\begin{lem}\label{lem:SingleChainTail}
There exist constants $\Cl{range},\Cl{srw}>0$ (which depend only on $d$) such that if $r \in (0,1]$ is small enough (depending on $d$), then the following results hold for all deterministic finite sequences $I=(I_\ell)_{\ell=1}^\nu$ in $\N$:
\begin{align*}
	\bP_r\Biggl( \sum_{\ell=1}^\nu \sigma_I(\ell,I_\ell)\le \Cr{range} \delta_d(r)^2\sum_{\ell=1}^\nu I_\ell \Biggr)
	\leq 4^{-\sum_{\ell=1}^\nu I_\ell},
\end{align*}
and for any $t>0$, 
\begin{align*}
	\bP_r\Biggl( \max_{\substack{1 \leq \ell \leq \nu\\0 \leq k \leq \sigma_I(\ell,I_\ell)}} \|S^{a_I(\ell)}_k\|_1 \geq \Cr{srw}\delta_d(r)t,\,\sum_{\ell=1}^\nu \sigma_I(\ell,I_\ell) \leq \delta_d(r)^2t \Biggr)
	\leq 4^{-t/\Cr{range}}.
\end{align*}
\end{lem}
\begin{proof}
Fix a deterministic finite sequence $I=(I_\ell)_{\ell=1}^\nu$ in $\N$ and let $J:=\sum_{\ell=1}^\nu I_\ell$, which is also regarded as a single-term sequence in $\N$.
Then, since each $\sigma_I(\ell,i)$ is a stopping time for $S_\cdot^{a_I(\ell)}$ and $S_\cdot^x$'s are independent, the chain of active frogs induced by $I$ and the trajectory of $S_\cdot^0$ up to time $\sigma_J(1,J)$ have the same law under $\bP_r$.
Hence, it suffices to prove that there exist constants $\Cr{range},\Cr{srw}>0$ (which depend only on $d$) such that if $r \in (0,1]$ is small enough (depending only on $d$), then
\begin{align}\label{eq:GeneralChain}
	\bP_r\bigl( \sigma_J(1,J) \leq \Cr{range} \delta_d(r)^2J \bigr)
	\leq 4^{-J},
\end{align}
and for any $t>0$, 
\begin{align}\label{eq:max}
	\bP_r\biggl( \max_{0 \leq k \leq \sigma_J(1,J)}\|S^0_k\|_1 \geq \Cr{srw}\delta_d(r)t,\,\sigma_J(1,J) \leq \delta_d(r)^2 t \biggr)
	\leq 4^{-t/\Cr{range}}.
\end{align}

Our first task is to prove that there exists a constant $p(d,\Cr{range}) \in (0,1)$ such that $\lim_{\Cr{range}\searrow 0}p(d,\Cr{range})=1$ and 
\begin{align}\label{eq:SD_Geo}
	\bP_r\biggl( \sigma_J(1,1) \ge \frac{\Cr{range}}{2}\lceil \delta_d(r)^2\rceil \biggr) \geq p(d,\Cr{range})
\end{align}
for all $r>0$. To lighten the notation, we write $t_1=\Cr{range}\lceil \delta_d(r)^2 \rceil/2$. 
Note that the event on the left-hand side is equivalent to that $\omega(\cdot)=0$ on $\mathcal{R}^0_{(0,t_1)}$. Thus by using Fubini's theorem and Jensen's inequality, we get
\begin{align*}
	\bP_r(\sigma_J(1,1)\ge t_1)
	\geq E\Bigl[ (1-r)^{\# \mathcal{R}^0_{(0,t_1)}} \Bigr]
	\geq (1-r)^{E[\# \mathcal{R}^0_{t_1}]}.
\end{align*}
It is known from \cite[Theorem~1]{DvoErd51} that there exists a constant $c$ (which depends only on $d$) such that 
\begin{equation*}
    E[\# \mathcal{R}^0_n]\le c \times
    \begin{dcases}
        \frac{n}{\log n}, & \text{if } d=2,\\
        n, & \text{if } d\ge 3.
    \end{dcases}
\end{equation*}
This together with the definition of $t_1$ shows that for all $r>0$, $(1-r)^{E[\# \mathcal{R}^0_{t_1}]}$ is bounded from below by some $p(d,\Cr{range}) \in (0,1)$ with $\lim_{\Cr{range}\searrow 0}p(d,\Cr{range})=1$.

Next, let us extend \eqref{eq:SD_Geo} to the joint distribution of $(\sigma_J(1,i)-\sigma_J(1,i-1))_{i=1}^J$.
Due to \eqref{eq:sigma1}, if $\mathcal{R}_{\sigma_J(1,i-1)}^0=\Gamma$ holds for some $\Gamma \subset \Z^d$, then the difference $\sigma_J(1,i)-\sigma_J(1,i-1)$ can be written as
\begin{equation*}
	\sigma_J(1,i)-\sigma_J(1,i-1)=\inf\bigl\{ k>0:S_{k+\sigma_J(1,i-1)}^0 \in \mathcal{O} \setminus \Gamma \bigr\}.
\end{equation*}
Moreover, to shorten notation, let $\rho(z,\Gamma):=\inf\{ k>0:S_k^z \in \mathcal{O} \setminus \Gamma \}$ and define for any $j=1,\dots,J$,
\begin{align*}
	\mathcal{E}_j:=\bigcap_{i=1}^j\{\sigma_J(1,i)-\sigma_J(1,i-1)\ge t_1\}.
\end{align*}
The strong Markov property implies that for any $j=1,\dots,J$, $z \in \Z^d$ and $\Gamma \subset \Z^d$,
\begin{align*}
	&P\Bigl(\bigl\{ S_{\sigma_J(1,j-1)}^0=z,\,\mathcal{R}_{\sigma_J(1,j-1)}^0=\Gamma \bigr\} \cap \mathcal{E}_j \Bigr)\\
	&= P(\rho(z,\Gamma) \geq t_1) \,P\Bigl( \bigl\{ S_{\sigma_J(1,j-1)}^0=z,\,\mathcal{R}_{\sigma_J(1,j-1)}^0=\Gamma \bigr\} \cap \mathcal{E}_{j-1} \Bigr).
\end{align*}
Since the last two probabilities are independent under $\P_r$, taking the expectation with respect to $\P_r$ in the above expression yields that
\begin{align*}
	&\bP_r\Bigl( \bigl\{ S_{\sigma_J(1,j-1)}^0=z,\,\mathcal{R}_{\sigma_J(1,j-1)}^0=\Gamma \bigr\} \cap \mathcal{E}_j \Bigr)\\
	&= \E_r\Bigl[ P(\rho(z,\Gamma) \geq t_1) \,P\Bigl( \bigl\{ S_{\sigma_J(1,j-1)}^0=z,\,\mathcal{R}_{\sigma_J(1,j-1)}^0=\Gamma \bigr\} \cap \mathcal{E}_{j-1} \Bigr) \Bigr]\\
	&= \E_r[P(\rho(z,\Gamma) \geq t_1)] \,\E_r\Bigl[ P\Bigl( \bigl\{ S_{\sigma_J(1,j-1)}^0=z,\,\mathcal{R}_{\sigma_J(1,j-1)}^0=\Gamma \bigr\} \cap \mathcal{E}_{j-1} \Bigr) \Bigr].
\end{align*}
By the fact that $z=S_{\sigma_J(1,j-1)}^0 \subset \mathcal{R}_{\sigma_J(1,j-1)}^0=\Gamma$ and $\rho(0,\Gamma') \ge \sigma_J(1,1)$ for all $\Gamma' \subset \Z^d$ containing $0$, the translation invariance of $(\omega,S)$ and \eqref{eq:SD_Geo} prove that
\begin{align*}
	\E_r[P(\rho(z,\Gamma) \geq t_1)]
	&= \E_r[P(\rho(0,\Gamma-z) \geq t_1)]\\
	&\geq \bP_r(\sigma_J(1,1) \geq t_1)\\
	&\geq p(d,\Cr{range}).
\end{align*}
Combining these considerations together with induction, we arrive at
\begin{align*}
	\bP_r(\mathcal{E}_J)
	&= \sum_{\substack{z \in \Z^d\\\Gamma \subset \Z^d}} P\Bigl(\bigl\{ S_{\sigma_J(1,J-1)}^0=z,\,\mathcal{R}_{\sigma_J(1,J-1)}^0=\Gamma \bigr\} \cap \mathcal{E}_J \Bigr)\\
	&\geq \sum_{\substack{z \in \Z^d\\\Gamma \subset \Z^d}} p(d,\Cr{range}) \,\bP_r\Bigl(\bigl\{ S_{\sigma_J(1,J-1)}^0=z,\,\mathcal{R}_{\sigma_J(1,J-1)}^0=\Gamma \bigr\} \cap \mathcal{E}_{J-1} \Bigr)\\
	&\geq p(d,\Cr{range}) \times \bP_r(\mathcal{E}_{J-1})\\
	&\geq p(d,\Cr{range})^J.
\end{align*}
This implies that the family $((\sigma_J(1,i)-\sigma_J(1,i-1))/t_1 \rceil)_{i=1}^J$ stochastically dominates a family of independent Bernoulli random variables  $(Y_i)_{i=1}^J$ with parameter $p(d,\Cr{range})$. 
Then, by the large deviation principle for the sum of Bernoulli random variables, we find that
\begin{align*}
	&\bP_r\bigl( \sigma_J(1,J) \leq \Cr{range} \delta_d(r)^2J \bigr)\\
  &\le \bP_r\biggl( \sum_{i=1}^J Y_i \leq \frac{J}{2} \biggr)
  \le \exp\left(-\frac{J}{2}\left(\log \frac{1}{2p(d,\Cr{range})}+\log \frac{1}{2(1-p(d,\Cr{range}))}+o(1)\right)\right)
\end{align*}
as $J\to\infty$. Since $\lim_{\Cr{range}\searrow 1}p(d,\Cr{range})=1$, we obtain the desired bound \eqref{eq:GeneralChain} by taking $\Cr{range}>0$ sufficiently small.

For the proof of \eqref{eq:max}, we recall the result obtained in \cite[Proposition~2.1.2-(b)]{LawLim10_book}:
there exists a constant $c$ (which depends only on $d$) such that for all $n \in \N_0$ and $s>0$,
\begin{align*}
	P\Bigl( \max_{0 \leq k \leq n} \|S_k^0\|_1 \geq s\sqrt{n} \Bigr) \leq c^{-1}e^{-cs^2}.
\end{align*}
Hence, for any $C,t>0$,
\begin{align*}
	&\bP_r\biggl( \max_{0 \leq k \leq \sigma_J(1,J)}\|S^0_k\|_1 \geq C\delta_d(r)t,\,\sigma_J(1,J) \leq \delta_d(r)^2 t \biggr)\\
	&\leq \bP_r\biggl( \max_{0 \leq k \leq \lfloor \delta_d(r)^2t \rfloor} \|S_k^0\|_1 \geq (C\sqrt{t})\sqrt{\lfloor \delta_d(r)^2t \rfloor} \biggr)
	\leq c^{-1}\exp\{ -cC^2t \},
\end{align*}
and \eqref{eq:max} follows by taking $C$ large enough (depending only on $c$ and $\Cr{range}$).
\end{proof}

We are now in a position to prove Theorem~\ref{thm:divergence}.

\begin{proof}[\bf Proof of Theorem~\ref{thm:divergence}]
Set $A_1:=(2\Cr{srw})^{-1}$ (which depends only on $d$).
Moreover, for any $x \in \Z^d \setminus \{0\}$ and $n \in \N$, denote by $v_n^x$ the closest point to $nx$ in $\mathcal{O}$ (which is chosen with a deterministic rule to break ties).
From \cite[Proposition~{2.4}-(2)]{Kub20}, $\bP_r(\cdot|\omega(0)=1)$-a.s., $T(0,v_n^x)/n$ converges to $\mu_r(x)$ as $n\to\infty$.
Therefore, if for each $x \in \Z^d \setminus \{0\}$, 
\begin{align}\label{eq:LowerBound}
	\liminf_{n \to \infty} \frac{1}{n}T(0,v_n^x) \geq A_1\delta_d(r)\|x\|_1,\qquad \bP_r(\cdot|\omega(0)=1) \hyphen \as,
\end{align}
then the statement of Theorem~\ref{thm:divergence} is true for all $x \in \Z^d \setminus \{0\}$.
Since $\mu_r(\cdot)$ and $\|\cdot\|_1$ are norms on $\R^d$ and $A_1$ is independent of $x$, we can easily extend the statement to the case where $x \in \R^d$.

Assume that $r \in (0,1]$ is small enough (depending only on $d$) to justify the argument below.
Fix $x \in \Z^d \setminus \{0\}$ and let $n \in \N$ be large enough.
Lemma~\ref{lem:T_chain}, together with $v_n^x \in \mathcal{O}$ and the shape theorem~\eqref{eq:shape}, implies that $\bP_r(\cdot|\omega(0)=1)$-a.s., there exists a (random) finite sequence $I=(I_\ell)_{\ell=1}^\nu$ in $\N$ such that
\begin{align*}
	T(0,v_n^x)=\sum_{\ell=1}^\nu \sigma_I(\ell,I_\ell).
\end{align*}
Now, if~\eqref{eq:LowerBound} does not hold, then this sequence, which depends on $n$, belongs to the following set for infinitely many $n$:
\begin{align*}
	\mathcal{I}(n):=\mleft\{I=(I_\ell)_{\ell=1}^\nu\colon
	\begin{minipage}{20.8em}
	\begin{itemize}
	\setlength{\itemsep}{0.2em}
	\setlength{\leftskip}{-2.1em}
	\item $\nu \in \N$,
	\item $\max_{1 \leq \ell \leq \nu}\max_{0 \leq k \leq \sigma_I(\ell,I_\ell)} \| S_k^{a_I(\ell)} \|_1 \geq n\|x\|_1/2$,
	\item $\sum_{\ell=1}^\nu \sigma_I(\ell,I_\ell) \leq A_1\delta_d(r)n\|x\|_1$
	\end{itemize}
	\end{minipage}
	\mright\}.
\end{align*}
We are going to make case distinction by using the events
\begin{align*}
	&\Cl[E]{E1}(n):=\biggl\{ \|v_n^x-nx\|_1 > \frac{n\|x\|_1}{2} \biggr\},\\
	&\Cl[E]{E2}(n):=\biggl\{ \exists I \in \mathcal{I}(n) \text{ such that } \sum_{\ell=1}^\nu I_\ell>A_1(\Cr{range}\delta_d(r))^{-1}n\|x\|_1 \biggr\},\\
	&\Cl[E]{E3}(n):=\biggl\{\exists I \in \mathcal{I}(n) \text{ such that } \sum_{\ell=1}^\nu I_\ell \leq A_1(\Cr{range}\delta_d(r))^{-1}n\|x\|_1 \biggr\}.
\end{align*}
Then we have the bound
\begin{align}\label{eq:Borel-Cantelli}
	\bP_r\bigl( T(0,v_n^x)<A_1\delta_d(r)n\|x\|_1 \big| \omega(0)=1 \bigr)
	\leq \bP_r(\Cr{E1}(n))+\bP_r(\Cr{E2}(n))+\bP_r(\Cr{E3}(n)).
\end{align}
Note that the condition $\omega(0)=1$ can be omitted on the right side of \eqref{eq:Borel-Cantelli} since $\Cr{E1}(n)$, $\Cr{E2}(n)$ and $\Cr{E3}(n)$ are independent of the values of $\omega(0)$.
Once we prove that $\bP_r(\Cr{E1}(n))$, $\bP_r(\Cr{E2}(n))$ and $\bP_r(\Cr{E3}(n))$ are summable in $n$, the desired result \eqref{eq:LowerBound} follows from \eqref{eq:Borel-Cantelli} and the Borel--Cantelli lemma.

First, since $\Cr{E1}(n)$ says that $\omega(y)=0$ holds for all $y \in B_1(nx,n\|x\|_1/2)$, there exists a constant $c$ (which depends only on $d$) such that
\begin{align*}
	\bP_r(\Cr{E1}(n)) \leq (1-r)^{cn^d\|x\|_1^d}.
\end{align*}
Hence, $\bP_r(\Cr{E1}(n))$ is summable in $n$ due to $\log(1-r)<0$.

Let us next estimate $\bP_r(\Cr{E2}(n))$ and $\bP_r(\Cr{E3}(n))$.
The event $\Cr{E2}(n)$ means that a chain of active frogs reaches a remote point in a relatively short time but meets a large number of sleeping frogs. It has a small probability since the density of the sleeping frogs tends to zero. 
On the other hand, the event $\Cr{E3}(n)$ means that a chain of active frogs meets a small number of frogs even though it reaches a remote point in a relatively short time. It has a small probability since the number of possible chains is limited but one of them has to move fast. 

In order to make the above observations into a proof, for any $M \in \N$, we define 
\begin{align*}
	\mathbb{I}_M:=\bigcup_{\nu=1}^M\biggl\{ (I_\ell)_{\ell=1}^\nu \in \N^\nu:\sum_{\ell=1}^\nu I_\ell=M \biggr\}
\end{align*}
and count the number of elements in $\mathbb{I}_M$:
\begin{align}\label{eq:ChoiceIndices}
	\#\mathbb{I}_M
	= \sum_{\nu=1}^M \binom{M-1}{\nu-1}
	\leq 2^M,
\end{align}
where, in the first equality, we used the fact that $I_\ell \geq 1$ for all $\ell \in [1,\nu]$.
Moreover, for simplicity of notation, let us write
\begin{align*}
	R:=A_1(\Cr{range}\delta_d(r))^{-1}\|x\|_1.
\end{align*}
When $\Cr{E2}(n)$ occurs, there exist an integer $M>Rn$ and a finite sequence $I=(I_\ell)_{\ell=1}^\nu \in \mathbb{I}_M$ such that $\sum_{\ell=1}^\nu \sigma_I(\ell,I_\ell) \leq \Cr{range} \delta_d(r)^2\sum_{\ell=1}^\nu I_\ell$.
Hence, the union bound, Lemma~\ref{lem:SingleChainTail} and \eqref{eq:ChoiceIndices} imply that
\begin{align*}
	\bP_r(\Cr{E2}(n))
	&\leq \sum_{M=\lceil Rn \rceil}^\infty \sum_{I \in \mathbb{I}_M} \bP_r\biggl( \sum_{\ell=1}^\nu \sigma_I(\ell,I_\ell) \leq \Cr{range} \delta_d(r)^2\sum_{\ell=1}^\nu I_\ell \biggr)\\
	&\leq \sum_{M=\lceil Rn \rceil}^\infty 2^{-M} \leq 2^{-Rn+1},
\end{align*}
and $\bP_r(\Cr{E2}(n))$ is summable in $n$.
On the other hand, on the event $\Cr{E3}(n)$, there exists $I \in \bigcup_{M=1}^{\lfloor Rn \rfloor} \mathbb{I}_M$ such that
\begin{align*}
	\max_{\substack{1 \leq \ell \leq \nu\\ 0 \leq k \leq \sigma_I(\ell,I_\ell)}} \|S^{a_I(\ell)}_k\|_1 \geq \frac{n\|x\|_1}{2},\qquad \sum_{\ell=1}^\nu \sigma_I(\ell,I_\ell) \leq A_1\delta_d(r)n\|x\|_1.
\end{align*}
Using \eqref{eq:ChoiceIndices} and Lemma~\ref{lem:SingleChainTail} with $t:=A_1\delta_d(r)^{-1}n\|x\|_1=\Cr{range}Rn$, one has
\begin{align*}
	&\bP_r(\Cr{E3}(n))\\
	&\leq \sum_{M=1}^{\lfloor Rn \rfloor}\sum_{I \in \mathbb{I}_M} \bP_r\Biggl( \max_{\substack{1 \leq \ell \leq \nu\\ 0 \leq k \leq \sigma_I(\ell,I_\ell)}} \|S^{a_I(\ell)}_k\|_1 \geq \Cr{srw}\delta_d(r)t,\,\sum_{\ell=1}^\nu \sigma_I(\ell,I_\ell) \leq \delta_d(r)^2t \Biggr)\\
	&\leq Rn2^{Rn} \times 4^{-t/\Cr{range}}=Rn2^{-Rn}.
\end{align*}
Therefore, $\bP_r(\Cr{E3}(n))$ is also summable in $n$, and the proof is complete.
\end{proof}


\section{Upper bound for the time constant}\label{sect:upper}
The aim of this section is to prove Theorem~\ref{thm:upper}, which says that for all $r \in (0,1]$ sufficiently small, the time constant $\mu_r(x)$ is bounded from above by $\delta_d(r)\|x\|_1$, up to a multiplicative constant.
For the proof, we extract an independent structure from the propagation of active frogs and use a percolation argument.
To carry out this, it is necessary to restrict the first passage time as follows: for any $A \subset \Z^d$ and $x,y \in \Z^d$,
\begin{align}\label{eq:restr_fpt}
	T_A(x,y):=\inf\mleft\{ \sum_{i=0}^{m-1} \tau(x_i,x_{i+1}):
	\begin{minipage}{12.8em}
		$m \in \N$, $x_0,x_1,\dots,x_{m-1} \in A$,\\
		$x_m \in \Z^d$ with $x_0=x$, $x_m=y$
	\end{minipage}
	\mright\}.
\end{align}
Note that $T_A(x,y)$ depends only on $\omega(z),S_\cdot^z$, $z \in A$.
Moreover, throughout this section, for any $u,v \in \Z^d$, $u \sim v$ (resp.~$u \overset{*}{\sim} v$) means $\|u-v\|_1=1$ (resp.~$\|u-v\|_\infty=1$).

We begin by defining a ``good'' propagation of active frogs that derives the desired factor $\delta_d(r)$.
Proposition~\ref{prop:good} below, which is the key ingredient in the proof of Theorem~\ref{thm:upper}, guarantees that the good propagation occurs with high probability.
Since the proof of Proposition~\ref{prop:good} is a little long, in Section~\ref{subsect:pf_upper}, we assume Proposition~\ref{prop:good} and prove Theorem~\ref{thm:upper} for now.
The proof of Proposition~\ref{prop:good} is given in Section~\ref{subsect:good} and consists of two stages.
The first stage involves activating many sleeping frogs in a large box: Section~\ref{subsect:sowing} explains how to sow seeds of active frogs, and Section~\ref{subsect:activating} carries out the activation of sleeping frogs by using the sown seeds.
In the second stage (Section~\ref{subsect:recursion}), we propagate the active frogs generated in the first stage outside the large box and recursively extend the activation of sleeping frogs into more remote regions.

\subsection{Percolation argument}\label{subsect:pf_upper}
The proof of Theorem~\ref{thm:upper} is based on a percolation construction defined as follows.
Let $\epsilon=\epsilon(d)$ and $\rho=\rho(d)$ be small and large positive constants to be chosen later (see \eqref{eq:eps} and \eqref{eq:rho} below), respectively.
Then, a site $v$ of $\Z^d$ is said to be \emph{$r$-good} if the following conditions are satisfied (see Figure~\ref{fig:good}):
\begin{itemize}
\item $B_\infty(\lceil r^{-d/2} \rceil v,r^{-(1/2+\epsilon)})$ contains an occupied site.
\item For any occupied site $x \in B_\infty(\lceil r^{-d/2} \rceil v,r^{-(1/2+\epsilon)})$ and $u \sim v$, there exists an occupied site $y \in B_\infty(\lceil r^{-d/2} \rceil u,r^{-(1/2+\epsilon)})$ such that
\begin{align*}
	T_{B_\infty(\lceil r^{-d/2} \rceil v,2\lceil r^{-d/2} \rceil)}(x,y)
	\leq \rho r^{-d/2} \delta_d(r).
\end{align*}
\end{itemize}
\begin{figure}
    \centering
    \includegraphics[width=0.6\linewidth]{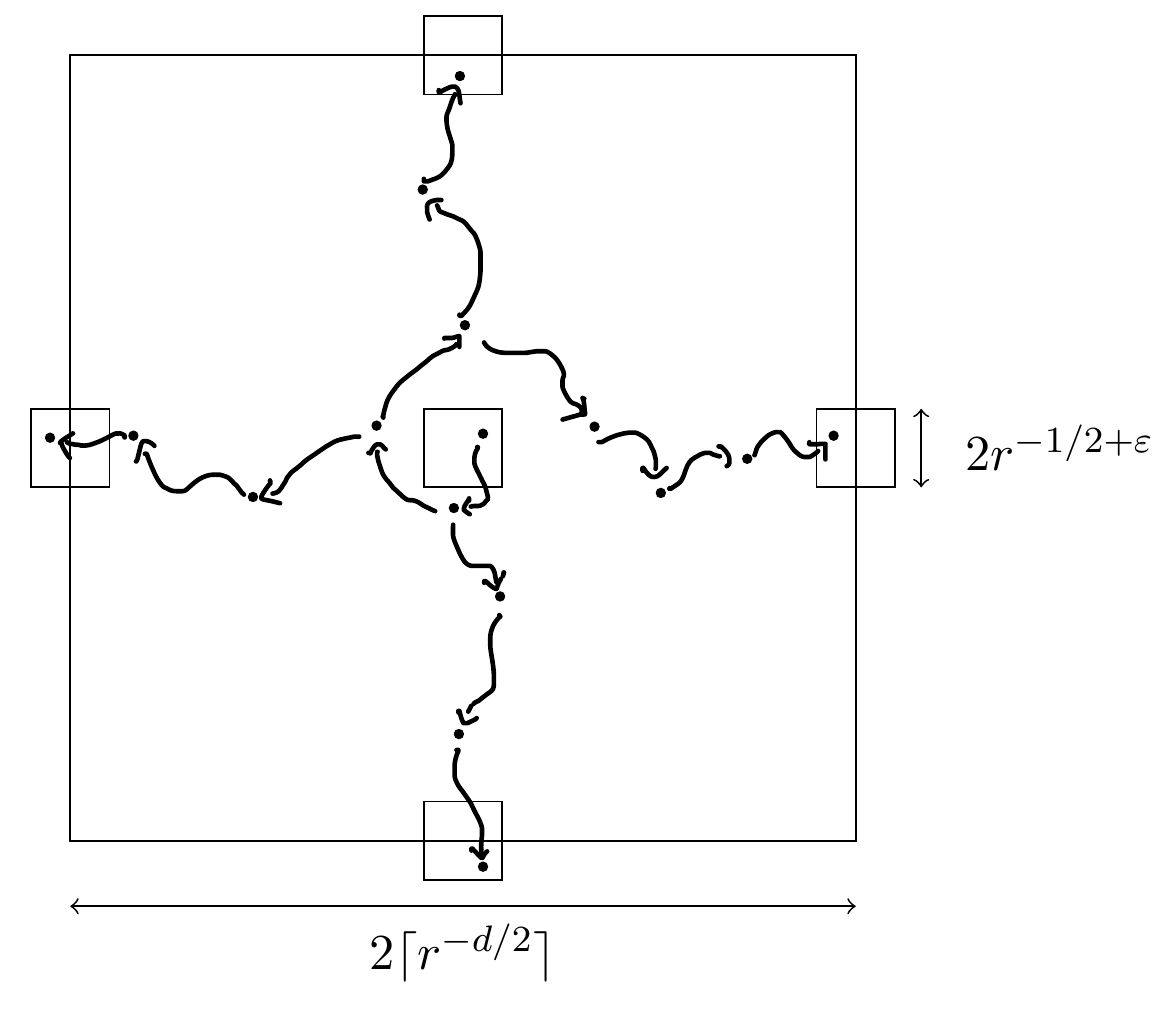}
    \caption{Schematic picture of the $r$-good event. Starting from any frog in the box of size $2r^{-(1/2+\epsilon)}$ at the center, we can reach any of the boxes in the coordinate directions in time comparable to $r^{-d/2} \delta_d(r)$.}
    \label{fig:good}
\end{figure}

It is generally difficult to extract an independent structure from the propagation of active frogs, since observing it requires keeping track of which frog activates which other.
The second condition above plays a crucial role in overcoming this difficulty.
Actually, the second condition can be interpreted as stating that \emph{any} frog starting from a specific box activates \emph{some} frog in each of the surrounding boxes.
Thus, the activation relation between two neighboring boxes can be disregarded, and we can extract an independent structure from the propagation of active frogs.

The key to the proof of Theorem~\ref{thm:upper} is the following proposition, which guarantees that every site $v$ is $r$-good with high probability when $r$ is small enough.

\begin{prop}\label{prop:good}
$\lim_{r \searrow 0} \sup_{v \in \Z^d} \bP_r(\text{$v$ is $r$-good})=1$ holds.
\end{prop}

We postpone the proof of Proposition~\ref{prop:good} to Section~\ref{subsect:good} and prove Theorem~\ref{thm:upper} assuming it. 
When a sequence $\gamma=(v_0,v_1,\dots,v_\ell)$ in $\Z^d$ satisfies that $v_0,v_1,\dots,v_\ell$ are $r$-good and $v_i \sim v_{i+1}$ for all $i \in [0,\ell-1]$, we call $\gamma$ an $r$-good path (from $v_0$ to $v_\ell$) and write $\#\gamma:=\ell$ for the length of $\gamma$.
Then, define for any $u,v \in \Z^d$,
\begin{align*}
	D_r(u,v):=\inf\{ \#\gamma:\text{$\gamma$ is an $r$-good path from $u$ to $v$} \},
\end{align*}
with the convention that $\inf\emptyset:=\infty$.


\begin{proof}[\bf Proof of Theorem~\ref{thm:upper}]
By the definition of $r$-good site, $(\1{\{ \text{$v$ is $r$-good} \}})_{v \in \Z^d}$ is a finitely dependent family of random variables taking values in $\{0,1\}$.
It follows from Proposition~\ref{prop:good} and a stochastic domination (see for instance \cite[Theorem~{7.65}]{Gri99_book}) that whenever $r$ is small enough (depending on $d$), $(\1{\{ \text{$v$ is $r$-good} \}})_{v \in \Z^d}$ stochastically dominates the independent Bernoulli site percolation on $\Z^d$ with parameter $(1+p_c)/2$ (where $p_c=p_c(d) \in (0,1)$ is the critical probability of independent Bernoulli site percolation on $\Z^d$).
This combined with \cite[(2.2) and Corollary~{2.2}]{GarMar10} implies that if $r$ is small enough (depending on $d$), then there exists a constant $C \geq 1$ (which depends only on $d$) such that for all $v \in \Z^d$ and $t \geq C\|v\|_\infty$,
\begin{align}\label{eq:chemical}
	\bP_r\Bigl( \inf\bigl\{ D_r(u,u'):u \in B_\infty(0,\sqrt{t}),\,u' \in B_\infty(v,\sqrt{t}) \bigr\} \geq t \Bigr)
	\leq C\exp\{ -\sqrt{t}/C \}.
\end{align}

Fix $x \in \Z^d \setminus \{0\}$ and let $r$ be small enough to establish \eqref{eq:chemical} for all $v \in \Z^d$ and $t \geq C\|v\|_\infty$.
In addition, write $R:=\lceil r^{-d/2} \rceil$ for simplicity of notation, and let $n \in \N$ be large enough to satisfy
\begin{align}\label{eq:tn}
	t_n:=2CR^{-1}n\|x\|_\infty \geq r^{-(1+2\epsilon)},\qquad n\|x\|_\infty \geq R.
\end{align}
Then, we consider the events
\begin{align*}
	&\Cl[G]{upper1}(n):=\Bigl\{ \inf\bigl\{ D_r(u,u'):u \in B_\infty(0,\sqrt{t_n}),\,u' \in B_\infty(\tilde{nx},\sqrt{t_n}) \bigr\}<t_n \Bigr\},\\
	&\Cl[G]{upper2}(n):=\biggl\{ \sup_{y \in B_\infty(0,2R\sqrt{t_n})} T(0,y) \leq t_n,\,\sup_{z \in B_\infty(R\tilde{nx},2R\sqrt{t_n}) \cap \mathcal{O}} T(z,nx) \leq t_n \biggr\},
\end{align*}
where $\tilde{nx}$ is the site of $\Z^d$ such that $nx \in B_\infty(R\tilde{nx},R)$ (which is chosen with a deterministic rule to break ties).
On the event $\Cr{upper1}(n)$, there exists an $r$-good path $(v_i)_{i=0}^\ell$ with $\ell<t_n$, $v_0 \in B_\infty(0,\sqrt{t_n})$ and $v_\ell \in B_\infty(\tilde{nx},\sqrt{t_n})$.
Hence, the definition of $r$-good site and the triangle inequality \eqref{eq:triangle} imply that for some occupied sites $y_i \in B_\infty(Rv_i,r^{-(1/2+\epsilon)})$, $0 \leq i \leq \ell$,
\begin{align*}
	T(y_0,y_\ell) \leq \sum_{i=0}^{\ell-1} T_{B_\infty(Rv_i,2R)}(y_i,y_{i+1}) \leq \rho r^{-d/2} \delta_d(r)t_n.
\end{align*}
Since, by \eqref{eq:tn},
\begin{align*}
	&y_0 \in B_\infty(Rv_0,r^{-(1/2+\epsilon)}) \subset B_\infty(0,2R\sqrt{t_n}),\\
	&y_\ell \in B_\infty(Rv_\ell,r^{-(1/2+\epsilon)}) \subset B_\infty(R\tilde{nx},2R\sqrt{t_n}),
\end{align*}
the triangle inequality \eqref{eq:triangle} implies that on $\Cr{upper1}(n) \cap \Cr{upper2}(n) \cap \{ \omega(0)=1 \}$,
\begin{align*}
	T(0,nx)
	&\leq T(0,y_0)+T(y_0,y_\ell)+T(y_\ell,nx)\\
	&\leq t_n+\rho r^{-d/2}\delta_d(r)t_n+t_n
	\leq 6C\rho \delta_d(r)n\|x\|_1.
\end{align*}
Therefore, once both $\bP_r(\Cr{upper1}(n)|\omega(0)=1)$ and $\bP_r(\Cr{upper2}(n)|\omega(0)=1)$ converge to one as $n \to \infty$, one has
\begin{align*}
	\lim_{n \to \infty}\bP_r\biggl( \frac{1}{n}T(0,nx) \leq 6C\rho\delta_d(r)\|x\|_1 \bigg| \omega(0)=1 \biggr)=1,
\end{align*}
and it follows by \eqref{eq:tc} that $\mu_r(x) \leq 6C\rho\delta_d(r)\|x\|_1$ holds for all $x \in \Z^d \setminus \{0\}$.
Since $\mu_r(\cdot)$ and $\|\cdot\|_1$ are norms on $\R^d$, this inequality can be easily extended to the case where $x \in \R^d$, and the theorem follows by taking $A_2:=6C\rho$. 

It remains to show that both $\bP_r(\Cr{upper1}(n)|\omega(0)=1)$ and $\bP_r(\Cr{upper2}(n)|\omega(0)=1)$ converge to one as $n \to \infty$.
Use \eqref{eq:tn} to obtain
\begin{align*}
	C\|\tilde{nx}\|_\infty \leq CR^{-1}(R+n\|x\|_\infty) \leq 2CR^{-1}n\|x\|_\infty=t_n.
\end{align*}
This together with \eqref{eq:chemical} implies that
\begin{align*}
	\lim_{n \to \infty}\bP_r(\Cr{upper1}(n)|\omega(0)=1) \geq  \lim_{n \to \infty}\bigl( 1-C\exp\{ -\sqrt{t_n}/C \} \bigr)=1.
\end{align*}
On the other hand, to estimate $\bP(\Cr{upper2}(n)|\omega(0)=1)$, we recall the result obtained in \cite[Proposition~{2.4}]{Kub19}:
there exist constants $\alpha,C'>0$ (which may depend on $r$) such that for all $y \in \Z^d$ and $t \geq C'\|y\|_1$,
\begin{align*}
	\bP_r(T(0,y) \geq t|\omega(0)=1) \leq C'\exp\{ -t^\alpha/C' \}.
\end{align*}
Since $\|z-nx\|_\infty \leq 3R\sqrt{t_n}$ if $z \in B_\infty(R\tilde{nx},2R\sqrt{t_n})$, it follows from the union bound, the translation invariance of $(\omega,S)$ and \eqref{eq:tc_bound} that for all large $n$,
\begin{align*}
	\bP_r(\Cr{upper2}(n)^c|\omega(0)=1)
	&\leq 2\sum_{y \in B_\infty(0,3R\sqrt{t_n})} \bP_r(T(0,y)>t_n|\omega(0)=1)\\
	&\leq 2C'(\# B_\infty(0,3R\sqrt{t_n}))\exp\{ -t_n^\alpha/C' \}.
\end{align*}
Clearly, the rightmost side converges to zero as $n \to \infty$, and hence $\bP_r(\Cr{upper2}(n)|\omega(0)=1)$ converges to one as $n \to \infty$.
\end{proof}

\subsection{Proof of Proposition~\ref{prop:good}}\label{subsect:good}
The aim of this subsection is to prove Proposition~\ref{prop:good}.
The proof relies on a recursive argument based on the intuition that if there are many active frogs in a certain box, then they activate many sleeping frogs around the box.
To this end, we first quote a result obtained in \cite{CanKubNak25_IP}.

\begin{prop}[{\!\!\cite[Lemma~{2.7}]{CanKubNak25_IP}}]\label{prop:CKN}
Let $0<\delta<1$.
There exists a constant $\Cr{CKN} \in (0,1)$ (which depends only on $d$ and $\delta$) such that for any integer $n \geq 2$ and for any finite subsets $A$ and $B$ of $\Z^d$ with $\max\{ \| x-y \|_2:x \in A,\,y \in B \} \leq \sqrt{n}$ and $\# B \geq \delta n^{d/2}$,
\begin{align}\label{eq:CKN_cor}
	P\Bigl( \#\bigl( \mathcal{R}_n^A \cap B \bigr)<\min\{ \Cr{CKN}\phi_d(n)\,\# A,(1-\delta)\,\# B \} \Bigr) \leq \exp\{ -\Cr{CKN}\,\# A \},
\end{align}
where $\phi_d(\cdot)$ is the function on $(0,\infty)$ given by
\begin{align*}
	\phi_d(t):=
	\begin{dcases*}
		\frac{t}{\log t}, & if $d=2$,\\
		t, & if $d \geq 3$.
	\end{dcases*}
\end{align*}
\end{prop}
We reproduce the proof of Proposition~\ref{prop:CKN} in the appendix for the convenience of the reader.

The aim of \cite{CanKubNak25_IP} is to estimate the upper large deviation probability of the first passage time.
Proposition~\ref{prop:CKN} is used there to successively generate a sufficient number of active frogs.
It also plays the same role in the present article, with a certain modification as stated in the following lemma.

\begin{lem}\label{lem:CKN_cor}
Let $0<\delta<1$.
There exists a constant $\Cl{CKN} \in (0,1)$ (which depends only on $d$ and $\delta$) such that for any integer $n \geq 2$ and for any finite subsets $A$ and $B$ of $\Z^d$ with $\max\{ \| x-y \|_2:x \in A,\,y \in B \} \leq \sqrt{n}$ and $\# B \geq \delta n^{d/2}$,
\begin{align*}
	&\bP_r\Bigl( \#(\mathcal{R}_n^A \cap B \cap \mathcal{O})<\frac{r}{2}\min\{ \Cr{CKN}\phi_d(n)\,\# A,(1-\delta)\,\# B \} \Bigr)\\
	&\leq \exp\{ -\Cr{CKN}\,\# A \}+\exp\Bigl\{ -\frac{r}{8}\min\{ \Cr{CKN}\phi_d(n)\,\# A,(1-\delta)\,\# B \} \Bigr\}.
\end{align*}
where $\phi_d(\cdot)$ is the function given in Proposition~\ref{prop:CKN}.
\end{lem}
\begin{proof}
Let $0<\delta<1$ and $2 \leq n \in \N$.
We now fix $r \in (0,1]$ and finite subsets $A,B$ of $\Z^d$ with $\max\{ \| x-y \|_2:x \in A,\,y \in B \} \leq \sqrt{n}$ and $\# B \geq \delta n^{d/2}$.
By the Chernoff bound for a sum of independent Bernoulli random variables,
\begin{align*}
	\P_r\Bigl( \#(\mathcal{R}_n^A \cap B \cap \mathcal{O})<\frac{r}{2}\#(\mathcal{R}_n^A \cap B) \Bigr)
	&= \P_r\biggl( \sum_{z \in \mathcal{R}_n^A \cap B} \omega(z)<\frac{r}{2} \#(\mathcal{R}_n^A \cap B) \biggr)\\
	&\leq \exp\Bigl\{ -\frac{r}{8} \#(\mathcal{R}_n^A \cap B) \Bigr\}.
\end{align*}
This, together with Proposition~\ref{prop:CKN} and the fact that $\mathcal{R}_n^A$ is independent of $\omega$, implies that
\begin{align*}
	&\bP_r\Bigl( \#(\mathcal{R}_n^A \cap B \cap \mathcal{O})<\frac{r}{2}\min\{ \Cr{CKN}\phi_d(n)\,\# A,(1-\delta)\,\# B \} \Bigr)\\
	&\leq \exp\{ -\Cr{CKN}\,\# A \}+E\biggl[ \1{\{ \#(\mathcal{R}_n^A \cap B) \geq \min\{ \Cr{CKN}\phi_d(n)\,\# A,(1-\delta)\,\# B \} \}} \exp\Bigl\{ -\frac{r}{8} \#(\mathcal{R}_n^A \cap B) \Bigr\} \biggr]\\
	&\leq \exp\{ -\Cr{CKN}\,\# A \}+\exp\Bigl\{ -\frac{r}{8}\min\{ \Cr{CKN}\phi_d(n)\,\# A,(1-\delta)\,\# B \} \Bigr\},
\end{align*}
which proves the lemma.
\end{proof}

The rest of this subsection consists of three parts.
The first two parts are for the initial step, and the third step is for the recursive argument.

\subsubsection{Sowing the seeds of active frogs}\label{subsect:sowing}
In order to construct a percolation structure on the mesoscopic scale, we specify the constant $\epsilon=\epsilon(d)$ appearing in the concept of $r$-good as follows:
\begin{align}\label{eq:eps}
	\epsilon:=\frac{1}{12d}.
\end{align}
Let us consider the boxes on the scale $r^{-(1/2+\epsilon)}$, which is much smaller than the scale $r^{-d/2}$ used in Section~\ref{subsect:pf_upper}.
In one of such boxes, we run the frogs up to time $r^{-(1+3\epsilon)} \gg (r^{-(1/2+\epsilon)})^2$ to activate some sleeping frogs in the neighboring boxes.
It might look strange that the time and distance ratio is much larger than the expected scaling of the time constant:
\begin{align*}
	\frac{r^{-(1+3\epsilon)}}{r^{-(1/2+\epsilon)}}=r^{-(1/2+2\epsilon)}\gg \delta_d(r).
\end{align*}
However, this does not cause a problem because the mesoscopic scale is used only to activate many sleeping frogs in the initial box and the extra time spent in the initial step will be wiped out in the recursion step.

Let us formally describe the mesoscopic percolation.
Define for each $v \in \Z^d$,
\begin{align}\label{eq:box_lamda}
\begin{split}
	&\Theta^\text{in}_r(v):=B_\infty\bigl( 7\lceil r^{-(1/2+\epsilon)} \rceil v,r^{-(1/2+\epsilon)} \bigr) \cap \Z^d,\\
	&\Theta_r(v):=B_\infty\bigl( 7\lceil r^{-(1/2+\epsilon)} \rceil v,2r^{-(1/2+\epsilon)} \bigr) \cap \Z^d,\\
	&\Theta^\text{out}_r(v):=\bigl(B_\infty\bigl( 7\lceil r^{-(1/2+\epsilon)} \rceil v,3r^{-(1/2+\epsilon)} \bigr) \cap \Z^d \bigr) \setminus \Theta_r(v).\\
\end{split}
\end{align}
The object of study here is the event $\mathcal{S}_r(v)$ that $\Theta_r^\text{in}(v)$ contains at least one occupied site and the following conditions hold for each occupied site $x \in \Theta_r^\text{in}(v)$ (see Figure~\ref{fig:Theta}):
\begin{itemize}
\item For any $u \overset{*}{\sim} v$, we can find an occupied site $a \in \Theta_r^\text{in}(u)$ with $T_{\Theta_r(v)}(x,a) \leq r^{-(1+3\epsilon)}$.

\item $T_{\Theta_r(v)}(x,b) \leq r^{-(1+3\epsilon)}$ holds for some occupied site $b \in \Theta_r^\text{out}(v)$.
\end{itemize}

The occurrence of $\mathcal{S}_r(v)$ is useful for sowing enough seeds of active frogs in and around the box $\Theta_r(v)$ within time $r^{-(1+3\epsilon)}$.
\begin{figure}
    \centering
    \includegraphics[width=0.75\linewidth]{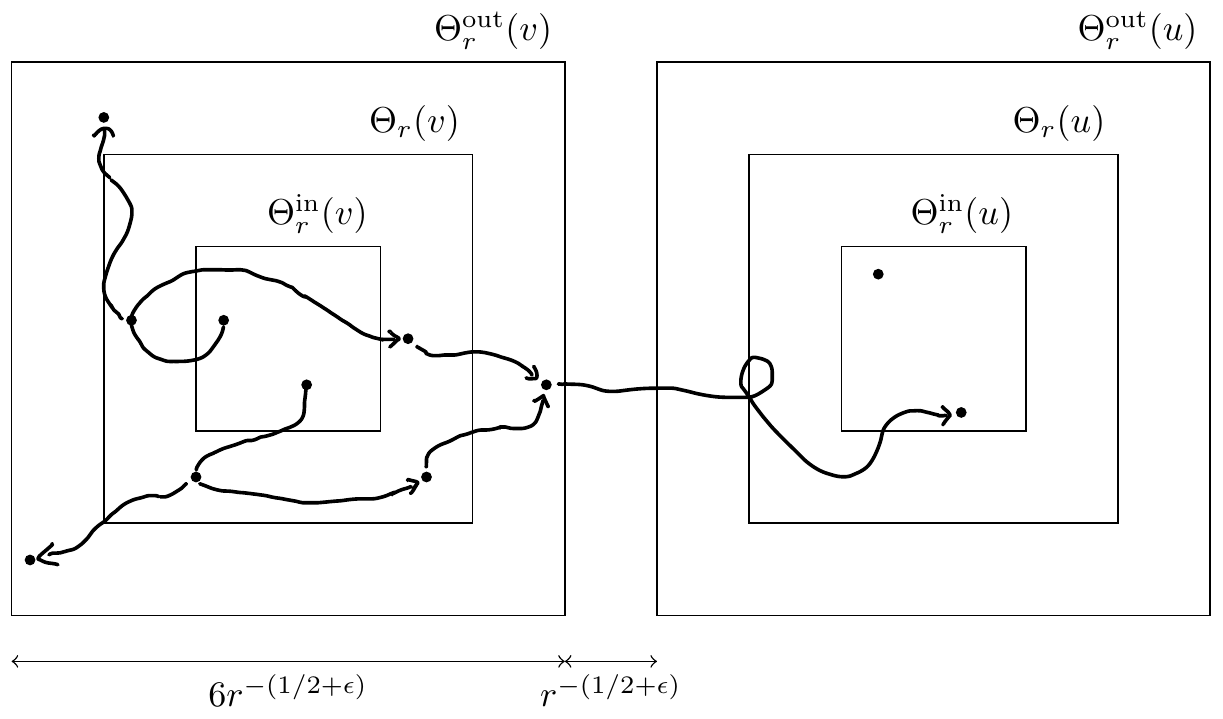}
    \caption{Schematic picture of the event $\mathcal{S}_r(v)$. Every frog in $\Theta_r^\text{in}(v)$ hits at least one frog in each of $\Theta_r^\text{out}(v)$ and $\Theta_r^\text{in}(u)$.}
    \label{fig:Theta}
\end{figure}
The following proposition says that each sowing event $\mathcal{S}_r(v)$ occurs with high probability when $r$ is small enough.

\begin{prop}\label{prop:sowing}
If $r$ is small enough (depending on $d$), then for any $v \in \Z^d$,
\begin{align*}
	\bP_r(\mathcal{S}_r(v)^c) \leq \exp\{ -r^{-\epsilon/6} \}.
\end{align*}
\end{prop}

Before proving Proposition~\ref{prop:sowing}, we prepare some notation and lemmata.
Let $n_r:=121d \lceil r^{-(1+2\epsilon)} \rceil$ and consider the events
\begin{align*}
	&\mathcal{S}_{r,1}:=\bigl\{ \#(\Theta_r^\text{in}(0) \cap \mathcal{O}) \geq 1 \bigr\},\\
	&\mathcal{S}_{r,2}:=\bigcap_{x \in \Theta_r^\text{in}(0)} \bigcap_{u \overset{*}{\sim} 0} \Bigl\{ \#\bigl( \mathcal{R}_{n_r}^{G_r(x)} \cap \Theta_r^\text{in}(u) \cap \mathcal{O} \bigr) \geq 1 \Bigr\},\\
	&\mathcal{S}_{r,3}:=\bigcap_{x \in \Theta_r^\text{in}(0)} \Bigl\{ \#\bigl( \mathcal{R}_{n_r}^{G_r(x)} \cap \Theta_r^\text{out}(0) \cap \mathcal{O} \bigr) \geq 1 \Bigr\},
\end{align*}
where $G_r(x):=(\mathcal{R}_{n_r}^x \cap \Theta_r(0) \cap \mathcal{O}) \setminus \{x\}$ denotes the set of all occupied sites, except for $x$, that are visited by the simple random walk $S_\cdot^x$ up to time $n_r$.
Let us first observe that the events $\mathcal{S}_{r,1}$, $\mathcal{S}_{r,2}$ and $\mathcal{S}_{r,3}$ cause the sowing event $\mathcal{S}_r(0)$ when $r$ is small enough.

\begin{lem}\label{lem:realize_sowing}
If $r$ is small enough (depending on $d$ and $\epsilon$), then
\begin{align*}
	\mathcal{S}_{r,1} \cap \mathcal{S}_{r,2} \cap \mathcal{S}_{r,3} \subset \mathcal{S}_r(0).
\end{align*}
\end{lem}
\begin{proof}
Note that the occurrence of $\mathcal{S}_{r,1}$ guarantees that $\Theta_r^\text{in}(0)$ contains at least one occupied site.
On $\mathcal{S}_{r,1} \cap \mathcal{S}_{r,2}$, for any occupied site $x \in \Theta_r^\text{in}(0)$ and $u \overset{*}{\sim} 0$, there exist occupied sites $z \in \Theta_r(0)$ and $a \in \Theta_r^\text{in}(u)$ such that $\tau(x,z) \vee \tau(z,a) \leq n_r$.
Hence, for all $r$ sufficiently small (depending on $d$),
\begin{align*}
	T_{\Theta_r(0)}(x,a) \leq \tau(x,z)+\tau(z,a) \leq 2n_r \leq r^{-(1+3\epsilon)}.
\end{align*}
Moreover, on $\mathcal{S}_{r,1} \cap \mathcal{S}_{r,3}$, for any occupied site $x \in \Theta_r^\text{in}(0)$, there exist occupied sites $z \in \Theta_r(0)$ and $b \in \Theta_r^\text{out}(0)$ such that $\tau(x,z) \vee \tau(z,b) \leq n_r$.
This implies that for all $r$ sufficiently small (depending on $d$),
\begin{align*}
	T_{\Theta_r(0)}(x,b) \leq \tau(x,z)+\tau(z,b) \leq 2n_r \leq r^{-(1+3\epsilon)}.
\end{align*}
Consequently, all the conditions in the definition of $\mathcal{S}_r(0)$ are satisfied, and therefore $\mathcal{S}_{r,1} \cap \mathcal{S}_{r,2} \cap \mathcal{S}_{r,3} \subset \mathcal{S}_r(0)$ holds if $r$ is small enough (depending on $d$).
\end{proof}

Our next task is to derive an upper bound for the cardinality of $G_r(x)$, which is an important factor of $\mathcal{S}_{r,2}$ and $\mathcal{S}_{r,3}$.

\begin{lem}\label{lem:occupied}
If $r$ is small enough (depending on $d$), then for any $x \in \Theta_r^\text{in}(0)$,
\begin{align*}
	\bP_r\bigl( \# G_r(x)<r^{-\epsilon/2} \bigr) \leq 3\exp\{ -r^{-\epsilon/5} \}.
\end{align*}
\end{lem}
\begin{proof}
We first recall the result obtained in \cite[Lemma~{3.1}]{AlvMacPop02}:
for any $\beta \in (0,1/2]$, if $n \in \N$ is large enough (depending on $d$ and $\beta$), then
\begin{align}\label{eq:AMP_range}
	P\bigl( \#(\mathcal{R}_n^0 \cap B_2(0,n^{1/2+\beta}))<n^{1-2\beta} \bigr) \leq 2\exp\{ -n^\beta \}.
\end{align}
From now on, take $\beta:=\epsilon/\{ 5(1+\epsilon) \}$ and let $r$ be small enough to justify the argument below.
Fix $x \in \Theta_r^\text{in}(0)$ and write $m_r:=\lceil r^{-(1+\epsilon)} \rceil$ for simplicity of notation.
Then, we have $(\mathcal{R}_{m_r}^x \cap B_2(x,m_r^{1/2+\beta}) \cap \mathcal{O}) \setminus \{x\} \subset G_r(x)$, and hence
\begin{align*}
	\#\bigl( \mathcal{R}_{m_r}^x \cap B_2(x,m_r^{1/2+\beta}) \cap \mathcal{O} \bigr)-1 \leq \# G_r(x).
\end{align*}
This, combined with the translation invariance of $(\omega,S)$ and the fact that $r^{-\epsilon/2}+1 \leq rm_r^{1-2\beta}/2$, shows that
\begin{align*}
	\bP_r\bigl( \# G_r(x)<r^{-\epsilon/2} \bigr)
	\leq \bP_r\biggl( \#\bigl( \mathcal{R}_{m_r}^0 \cap B_2(0,m_r^{1/2+\beta}) \cap \mathcal{O} \bigr)<\frac{r}{2}m_r^{1-2\beta} \biggr).
\end{align*}
Hence, our task is to prove that the right side is bounded from above by $\exp\{ -r^{-\epsilon/5} \}$.
Thanks to \eqref{eq:AMP_range},
\begin{align*}
	P\bigl( \#(\mathcal{R}_{m_r}^0 \cap B_2(0,m_r^{1/2+\beta}))<m_r^{1-2\beta} \bigr)
	\leq 2\exp\{ -m_r^\beta \}.
\end{align*}
Moreover, the Chernoff bound for a sum of independent Bernoulli random variables tells us that on the event that $\#(\mathcal{R}_{m_r}^0 \cap B_2(0,m_r^{1/2+\beta})) \geq m_r^{1-2\beta}$,
\begin{align*}
	\P_r\biggl( \#\bigl( \mathcal{R}_{m_r}^0 \cap B_2(0,m_r^{1/2+\beta}) \cap \mathcal{O} \bigr)<\frac{r}{2}m_r^{1-2\beta} \biggr)
	\leq \exp\Bigl\{ -\frac{r}{8}m_r^{1-2\beta} \Bigr\}.
\end{align*}
With these observations, noting that $m_r^\beta \geq r^{-\epsilon/5}$ and $rm_r^{1-2\beta}/8 \geq r^{-\epsilon/5}$, one has
\begin{align*}
	&\bP_r\biggl( \#\bigl( \mathcal{R}_{m_r}^0 \cap B_2(0,m_r^{1/2+\beta}) \cap \mathcal{O} \bigr)<\frac{r}{2}m_r^{1-2\beta} \biggr)\\
	&\leq 2\exp\{ -m_r^\beta \}+\exp\Bigl\{ -\frac{r}{8}m_r^{1-2\beta} \Bigr\}
	\leq 3\exp\{ -r^{-\epsilon/5} \},
\end{align*}
and the lemma follows.
\end{proof}

\begin{rem}
Actually, \cite[Lemma~{3.1}]{AlvMacPop02} proves \eqref{eq:AMP_range} with $2\beta$ replaced by $\beta$ in $d \geq 3$.
However, the argument used there also works for all $d \geq 2$ by modifying the statement of \cite[Lemma~{3.1}]{AlvMacPop02} as in \eqref{eq:AMP_range}.
\end{rem}

We are now in a position to prove Proposition~\ref{prop:sowing}.

\begin{proof}[\bf Proof of Proposition~\ref{prop:sowing}]
By the translation invariance of $(\omega,S)$, it suffices to prove the proposition in the case where $v=0$.
Thanks to Lemma~\ref{lem:realize_sowing}, our task is to prove that if $r$ is small enough (depending on $d$), then
\begin{align}\label{eq:good}
	\bP_r(\mathcal{S}_{r,i}^c) \leq \frac{1}{3}\exp\{ -r^{-\epsilon/6} \},\qquad i=1,2,3.
\end{align}

Assume that $r$ is small enough to justify the argument below.
The bound for $\bP_r(\mathcal{S}_{r,1}^c)$ immediately follows from the independence of $\omega$ and the fact that $(1-t)^{1/t} \leq e^{-1}$ holds for all $t \in (0,1]$ (note that this fact is often used below without further comment):
\begin{align*}
	\bP_r(\mathcal{S}_{r,1}^c)
	\leq (1-r)^{r^{-d(1/2+\epsilon)}}
	\leq \exp\{ -r^{1-d(1/2+\epsilon)} \}
	\leq \frac{1}{3}\exp\{ -r^{-\epsilon/6} \}.
\end{align*}

We next estimate $\bP_r(\mathcal{S}_{r,2}^c)$.
The union bound and Lemma~\ref{lem:occupied} imply that
\begin{align}\label{eq:S2}
\begin{split}
	\bP_r(\mathcal{S}_{r,2}^c)
	&\leq 3^{2d+1}r^{-d(1/2+\epsilon)}\exp\{ -r^{-\epsilon/6} \}\\
	&\quad +\sum_{x \in \Theta_r^\text{in}(0)}\sum_{u \overset{*}{\sim} 0} \bP_r\Bigl( \# G_r(x) \geq r^{-\epsilon/2},\, \#\bigl( \mathcal{R}_{n_r}^{G_r(x)} \cap \Theta_r^\text{in}(u) \cap \mathcal{O} \bigr)=0 \Bigr).
\end{split}
\end{align}
Since $G_r(x) \subset \Theta_r(0) \setminus \{x\}$ (in particular, $G_r(x)$ depends only on $S_\cdot^x$ and $\omega(z)$, $z \in \Theta_r(0) \setminus \{x\}$), each summand of the last term in \eqref{eq:S2} can be rewritten as follows:
\begin{align}\label{eq:S2_rewritten}
\begin{split}
	&\bP_r\Bigl( \# G_r(x) \geq r^{-\epsilon/2},\, \#\bigl( \mathcal{R}_{n_r}^{G_r(x)} \cap \Theta_r^\text{in}(u) \cap \mathcal{O} \bigr)=0 \Bigr)\\
	&= \sum_{\substack{A \subset \Theta_r(0) \setminus \{x\}\\ \# A \geq r^{-\epsilon/2}}} \bP_r(G_r(x)=A) \,\bP_r\Bigl( \#\bigl( \mathcal{R}_{n_r}^A \cap \Theta_r^\text{in}(u) \cap \mathcal{O} \bigr)=0 \Bigr).
\end{split}
\end{align}
Note that for any $A \subset \Theta_r(0)$,
\begin{align*}
	\max\bigl\{ \|y-y'\|_2:y \in A,\,y' \in \Theta_r^\text{in}(u) \bigr\} \leq 11\sqrt{d}r^{-(1/2+\epsilon)} \leq \sqrt{n_r}
\end{align*}
and $\#\Theta_r^\text{in}(u) \geq \delta n_r^{d/2}$ for some $\delta \in (0,1)$ (which depends only on $d$).
Hence, Lemma~\ref{lem:CKN_cor} (with the $\delta$ above) yields that for all $A \subset \Theta_r(0) \setminus \{x\}$ with $\# A \geq r^{-\epsilon}$,
\begin{align*}
	&\bP_r\Bigl( \#\bigl( \mathcal{R}_{n_r}^A \cap \Theta_r^\text{in}(u) \cap \mathcal{O} \bigr)=0 \Bigr)\\
	&\leq \bP_r\biggl( \#\bigl( \mathcal{R}_{n_r}^A \cap \Theta_r^\text{in}(u) \cap \mathcal{O} \bigr)<\frac{r}{2}\min\bigl\{ \Cr{CKN}\phi_d(n_r)\,\# A,(1-\delta)\,\#\Theta_r^\text{in}(u) \bigr\} \biggr)\\
	&\leq \exp\{ -\Cr{CKN}r^{-\epsilon/2} \}+\exp\Bigl\{ -\frac{r}{8}\min\bigl\{ \Cr{CKN}\phi_d(n_r)r^{-\epsilon/2},(1-\delta)\delta n_r^{d/2} \bigr\} \Bigr\}.
\end{align*}
The rightmost side is smaller than or equal to $2\exp\{ -\Cr{CKN}r^{-\epsilon} \}$, and it follows by \eqref{eq:S2_rewritten} that
\begin{align*}
	\bP_r\Bigl( \# G_r(x) \geq r^{-\epsilon},\, \#\bigl( \mathcal{R}_{n_r}^{G_r(x)} \cap \Theta_r^\text{in}(u) \cap \mathcal{O} \bigr)=0 \Bigr)
	\leq 2\exp\{ -\Cr{CKN}r^{-\epsilon/2} \}.
\end{align*}
Substituting this into \eqref{eq:S2} proves
\begin{align*}
	\bP_r(\mathcal{S}_{r,2}^c)
	\leq 3^{2d+1}r^{-d(1/2+\epsilon)}\bigl( \exp\{ -r^{-\epsilon/5} \}+2\exp\{ -\Cr{CKN}r^{-\epsilon/2} \} \bigr)
	\leq \frac{1}{3} \exp\{ -r^{-\epsilon/6} \},
\end{align*}
which is the desired bound for $\bP_r(\mathcal{S}_{r,2}^c)$.

Let us finally estimate $\bP_r(\mathcal{S}_{r,3}^c)$.
Note that $\Theta_r(0)$ and $\Theta_r^\text{out}(0)$ are disjoint.
Furthermore, it is easy to see that if $A \subset \Theta_r(0) \setminus \{x\}$ with $\# A \geq r^{-\epsilon/2}$, then
\begin{align*}
	\max\bigl\{ \|y-y'\|_2:y \in A,\,y' \in \Theta_r^\text{out}(0) \bigr\} \leq 5\sqrt{d}r^{-(1/2+\epsilon)} \leq \sqrt{n_r}
\end{align*}
and $\#\Theta_r^\text{out}(0) \geq \delta'n_r^{d/2}$ for some $\delta' \in (0,1)$ (which depends only on $d$).
Hence, we can apply the same argument used to estimate $\bP_r(\mathcal{S}_{r,2}^c)$ and obtain for some $c \in (0,1)$ (which depends only on $d$),
\begin{align*}
	\bP_r(\mathcal{S}_{r,3}^c)
	\leq 4^dr^{-d(1/2+\epsilon)}\bigl( \exp\{ -r^{-\epsilon/5} \}+2\exp\{ -cr^{-\epsilon/2} \} \bigr)
	\leq \frac{1}{3} \exp\{ -r^{-\epsilon/6} \}.
\end{align*}
Therefore, \eqref{eq:good} is proved and the proposition follows.
\end{proof}

\subsubsection{Activating many sleeping frogs in the initial box}\label{subsect:activating}
In order to activate many sleeping frogs in the initial box, we introduce the event $\mathcal{A}_r$ that $\Theta_r^\text{in}(0)$ contains at least one occupied site, and for every occupied site $x \in \Theta_r^\text{in}(0)$,
\begin{align*}
	\max_{y \in \Lambda_r} T_{\Lambda_r}(x,y) \leq 5dr^{-(d+1)/2},
\end{align*}
where $\Lambda_r:=B_\infty\bigl( 0,r^{-(d+1)/4} \bigr) \cap \Z^d$.
Roughly speaking, the occurrence of $\mathcal{A}_r$ means that when an active frog starts moving from $\Theta_r^\text{in}(0)$, all sleeping frogs in $\Lambda_r$ are activated within time $5dr^{-(d+1)/2}$ (see Figure~\ref{fig:Lambda}).
Our main task here is to prove the following proposition, which guarantees that the activating event $\mathcal{A}_r$ occurs with high probability when $r$ is small enough.
\begin{figure}
\centering
\includegraphics[width=0.5\linewidth]{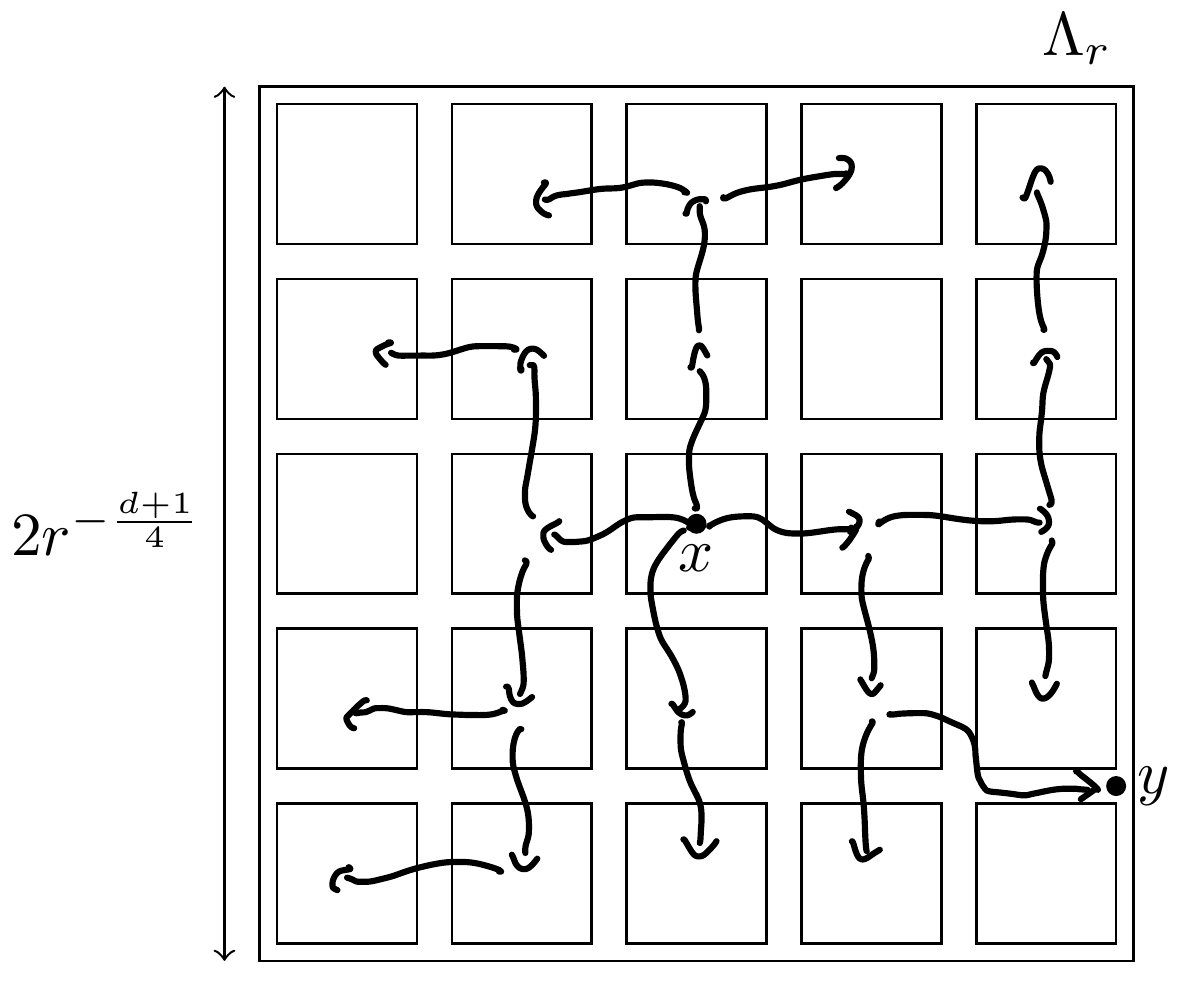}
\caption{Schematic picture of $\mathcal{A}_r$. Smaller boxes are $\Theta_r^\text{out}(v)$ and the frogs reach many of them in a relatively short time. Then from one of those boxes, we can find a frog that reaches $y$.}
\label{fig:Lambda}
\end{figure}

\begin{prop}\label{prop:activating}
$\lim_{r \searrow 0} \bP_r(\mathcal{A}_r)=1$ holds.
\end{prop}

To prove Proposition~\ref{prop:activating}, we begin by preparing some notation and lemmata.
Set $V_r:=B_\infty(0,r^{-(d-1)/4+2\epsilon}) \cap \Z^d$ and define for any $x \in \Theta_r^\text{in}(0)$,
\begin{align*}
	W_r(x):=\Biggl\{ w \in \bigcup_{v \in V_r} \Theta_r^\text{out}(v): T_{\bigcup_{v \in V_r} \Theta_r(v)}(x,w) \leq 2r^{-(d+1)/2} \Biggr\} \cap \mathcal{O}.
\end{align*}
This represents the seeds of active frogs that originate from the active frogs starting at $x$, and it is important in the proof of Proposition~\ref{prop:activating} that $W_r(x)$ depends only on $\omega$ and $S_\cdot^z$, $z \in \bigcup_{v \in V_r} \Theta_r(v)$.
We now consider the events
\begin{align*}
	&\mathcal{A}_{r,1}:=\bigcap_{v \in V_r} \mathcal{S}_r(v),\\
	&\mathcal{A}_{r,2}:=\bigcap_{x \in \Theta_r^\text{in}(0)} \Bigl\{ \max_{y \in \Lambda_r} \min_{w \in W_r(x)} H(w,y) \leq 4dr^{-(d+1)/2} \Bigr\},
\end{align*}
where $H(w,y):=\inf\{ k \geq 0:S_k^w=y \}$ stands for the hitting time of $y$ for the simple random walk $S_\cdot^w$.

As stated in Lemma~\ref{lem:W} below, the event $\mathcal{A}_{r,1}$ is useful to generate a lot of active frogs in $\Lambda_r$.
On the other hand, the event $\mathcal{A}_{r,2}$ ensures that the active frogs generated in $\Lambda_r$ can reach any point in $\Lambda_r$.

\begin{lem}\label{lem:W}
Let $r \in (0,1]$ be small enough to have $\bigcup_{v \in V_r} \Theta_r^\mathrm{out}(v) \subset \Lambda_r$.
Then, on the event $\mathcal{A}_{r,1}$, $\# W_r(x) \geq r^{-d\{ (d-1)/4-2\epsilon \}}$ holds for any occupied site $x \in \Theta_r^\mathrm{in}(0)$.
\end{lem}
\begin{proof}
Assume that $r$ is small enough to have $\bigcup_{v \in V_r} \Theta_r^\mathrm{out}(v) \subset \Lambda_r$ and $\mathcal{A}_{r,1}$ occurs.
Fix an occupied site $x \in \Theta_r^\text{in}(0)$ (note that on $\mathcal{A}_{r,1}$, we have at least one occupied site in $\Theta_r^\text{in}(0)$).
Then, for any $u \in V_r$, there exist a sequence $(u_i,a_i)_{i=0}^\ell$ (with $\ell:=\|u\|_\infty$) and an occupied site $b \in \Theta_r^\text{out}(u)$ satisfying the following conditions:
\begin{itemize}
\item $u_0,\dots,u_\ell$ are distinct points in $V_r$ such that $u_0=0$, $u_\ell=u$ and $u_i \overset{*}{\sim} u_{i+1}$ for each $i \in [0,\ell-1]$.

\item $a_0=x$ and $a_i \in \Theta_r^\text{in}(u_i) \cap \mathcal{O}$ for each $i \in [1,\ell]$.

\item $T_{\Theta_r(u_i)}(a_i,a_{i+1}) \leq r^{-(1+3\epsilon)}$ for each $i \in [0,\ell-1]$ and $T_{\Theta_r(u)}(a_\ell,b) \leq r^{-(1+3\epsilon)}$.
\end{itemize}
Hence, due to $\ell \leq r^{-(d-1)/4+2\epsilon}$ and $\epsilon \leq 1/4$ (see \eqref{eq:eps}),
\begin{align*}
	T_{\bigcup_{v \in V_r} \Theta_r(v)}(x,b)
	&\leq \sum_{i=0}^{\ell-1} T_{\Theta_r(u_i)}(a_i,a_{i+1})+T_{\Theta_r(u)}(a_\ell,b)\\
	&\leq (\ell+1)r^{-(1+3\epsilon)}
	\leq 2r^{-(d-1)/4+2\epsilon-(1+3\epsilon)}
	\leq 2r^{-(d+1)/2}.
\end{align*}
This implies that
\begin{align*}
	\#\Bigl( \Bigl\{ w \in \Theta_r^\text{out}(u): T_{\bigcup_{v \in V_r} \Theta_r(v)}(x,w) \leq 2r^{-(d+1)/2} \Bigr\} \cap \mathcal{O} \Bigr)
	\geq \#\{ b \} \geq 1.
\end{align*}
Since $\Theta_r^\text{out}(u)$'s are disjoint, it follows that for any occupied site $x \in \Theta_r^\text{in}(0)$,
\begin{align*}
	\# W_r(x) \geq \# V_r \geq r^{-d\{ (d-1)/4-2\epsilon \}},
\end{align*}
and the proof is complete.
\end{proof}

The following lemma is used to realize $\mathcal{A}_{r,2}$.

\begin{lem}\label{lem:all_visit}
If $r \in (0,1]$ is small enough (depending on $d$), then for any $A \subset \Lambda_r$ with $\# A \geq r^{-d\{ (d-1)/4-2\epsilon \}}$,
\begin{align*}
	P\Bigl( \max_{y \in \Lambda_r} \min_{w \in A} H(w,y)>4dr^{-(d+1)/2} \Bigr)
	\leq \exp\{ -r^{-1/4} \}.
\end{align*}
\end{lem}
\begin{proof}
The union bound and the independence of $S$ show that for any $A \subset \Lambda_r$ with $\# A \geq r^{-d\{ (d-1)/4-2\epsilon \}}$,
\begin{align}\label{eq:all_visit}
\begin{split}
	&P\Bigl( \max_{y \in \Lambda_r} \min_{w \in A} H(w,y)>4dr^{-(d+1)/2} \Bigr)\\
	&\leq \sum_{y \in \Lambda_r \setminus A} \prod_{w \in A} \bigl\{ 1-P\bigl( H(w,y) \leq 4dr^{-(d+1)/2} \bigr) \bigr\}.
\end{split}
\end{align}
To estimate the probabilities of the right side in \eqref{eq:all_visit}, we recall the result obtained in \cite[Theorem~{2.2}-(ii)]{AlvMacPop02}:
there exists a constant $c \in (0,1)$ (which depends only on $d$) such that for any $z \in \Z^d \setminus \{0\}$ and $n \geq \|z\|_2^2$,
\begin{align}\label{eq:AMP}
	P(H(0,z) \leq n) \geq c \times
	\begin{dcases*}
		\frac{1}{\log(1+\|z\|_2)}, & if $d=2$,\\[0.5em]
		\|z\|_2^{-d+2}, & if $d \geq 3$.
	\end{dcases*}	
\end{align}
This, combined with the translation invariance of $S$, implies that if $r$ is small enough (depending on $d$), then there exists a constant $c' \in (0,1)$ (which depends only on $d$) such that for all $y \in \Lambda_r \setminus A$ and $w \in A$,
\begin{align*}
	P\bigl( H(w,y) \leq 4dr^{-(d+1)/2} \bigr) \geq c' \times
	\begin{dcases*}
		\frac{1}{|\log r|}, & if $d=2$,\\[0.5em]
		r^{(d+1)(d-2)/4}, & if $d \geq 3$.
	\end{dcases*}
\end{align*}
Therefore, in the case where $d=2$, for all $r$ sufficiently small (depending on $d$), the right side of \eqref{eq:all_visit} is bounded from above by
\begin{align*}
	&\bigl( 2r^{-3/4}+1 \bigr)^2 \Bigl( 1-\frac{c'}{|\log r|} \Bigr)^{r^{-1/2+4\epsilon}}\\
	&\leq \bigl( 2r^{-3/4}+1 \bigr)^2 \exp\biggl\{ -\frac{c'r^{-1/2+4\epsilon}}{|\log r|} \biggr\}
	\leq \exp\{ -r^{-1/4} \}.
\end{align*}
Here we used the fact that $\epsilon=1/24$ (see \eqref{eq:eps}) in the last inequality.
Similarly, in the case where $d \geq 3$, for all $r$ sufficiently small (depending on $d$), the right side of \eqref{eq:all_visit} is smaller than or equal to
\begin{align*}
	&\bigl( 2r^{-(d+1)/4}+1 \bigr)^d \Bigl( 1-c'r^{(d+1)(d-2)/4} \Bigr)^{r^{-d\{ (d-1)/4-2\epsilon \}}}\\
	&\leq \bigl( 2r^{-(d+1)/4}+1 \bigr)^d \exp\{ -c'r^{-1/2+2d\epsilon} \}
	\leq \exp\{ -r^{-1/4} \},
\end{align*}
and the proof is complete.
\end{proof}

Let us move on to the proof of Proposition~\ref{prop:activating}.

\begin{proof}[\bf Proof of Proposition~\ref{prop:activating}]
Let $r$ be small enough to satisfy $\bigcup_{v \in V_r} \Theta_r^\text{out}(v) \subset \Lambda_r$.
Then, Lemma~\ref{lem:W} implies that on the event $\mathcal{A}_{r,1} \cap \mathcal{A}_{r,2}$, for any occupied site $x \in \Theta_r^\text{in}(0)$ and for any $y \in \Lambda_r$, there exists $w \in \bigcup_{v \in V_r} \Theta_r^\text{out}(v) \subset \Lambda_r$ such that
\begin{align*}
	T_{\Lambda_r}(x,y)
	&\leq T_{\bigcup_{v \in V_r} \Theta_r(v)}(x,w)+H(w,y)\\
	&\leq 2r^{-(d+1)/2}+4dr^{-(d+1)/2} \leq 5dr^{-(d+1)/2}.
\end{align*}
Hence, $\mathcal{A}_{r,1} \cap \mathcal{A}_{r,2} \subset \mathcal{A}_r$ holds, and our task is to show that
\begin{align}\label{eq:activating}
	\lim_{r \searrow 0} \bP_r((\mathcal{A}_{r,1} \cap \mathcal{A}_{r,2})^c)=0.
\end{align}

First, we use the union bound and Proposition~\ref{prop:sowing} to obtain for all $r$ sufficiently small (depending on $d$),
\begin{align*}
	\bP_r(\mathcal{A}_{r,1}^c)
	\leq \sum_{v \in V_r} \bP_r(\mathcal{S}_r(v)^c)
	\leq \bigl( 2r^{-(d-1)/4+2\epsilon}+1 \bigr)^d \exp\{ -r^{-\epsilon/6} \},
\end{align*}
which yields $\lim_{r \searrow 0} \bP_r(\mathcal{A}_{r,1}^c)=0$.
Next, the union bound and Lemma~\ref{lem:W} prove that if $r$ is small enough to satisfy $\bigcup_{v \in V_r} \Theta_r^\text{out}(v) \subset \Lambda_r$, then
\begin{align}\label{eq:A2}
\begin{split}
	&\bP_r(\mathcal{A}_{r,1} \cap \mathcal{A}_{r,2}^c)\\
	&\leq \sum_{x \in \Theta_r^\text{in}(0)} \bP_r\mleft(
		\begin{minipage}{18.5em}
			$\# W_r(x) \geq r^{-d\{ (d-1)/4-2\epsilon \}}$ and\\
			$\max_{y \in \Lambda_r} \min_{w \in W_r(x)} H(w,y)>4dr^{-(d+1)/2}$
		\end{minipage}
		\mright).
\end{split}
\end{align}
Noting that $W_r(x)$ depends only on $\omega$ and $S_\cdot^z$, $z \in \bigcup_{v \in V_r} \Theta_r(v)$, one can apply Lemma~\ref{lem:all_visit} to estimate each summand in \eqref{eq:A2}:
\begin{align*}
	&\bP_r\mleft(
	\begin{minipage}{18.5em}
		$\# W_r(x) \geq r^{-d\{ (d-1)/4-2\epsilon \}}$ and\\
		$\max_{y \in \Lambda_r} \min_{w \in W_r(x)} H(w,y)>4dr^{-(d+1)/2}$
	\end{minipage}
	\mright)\\
	&= \sum_{\substack{A \subset \bigcup_{v \in V_r} \Theta_r^\text{out}(v)\\ \# A \geq r^{-d\{ (d-1)/4-2\epsilon \}}}} \bP_r(W_r(x)=A) \,P\Bigl( \max_{y \in \Lambda_r} \min_{w \in A} H(w,y)>4dr^{-(d+1)/2} \Bigr)\\
	&\leq \exp\{ -r^{-1/4} \}.
\end{align*}
Substituting this into \eqref{eq:A2} yields that for all $r$ sufficiently small (depending on $d$),
\begin{align*}
	\bP_r(\mathcal{A}_{r,1} \cap \mathcal{A}_{r,2}^c)
	\leq \bigl( 2r^{-(1/2+\epsilon)}+1 \bigr)^d \exp\{ -r^{-1/4} \},
\end{align*}
which implies $\lim_{r \searrow 0} \bP_r(\mathcal{A}_{r,1} \cap \mathcal{A}_{r,2}^c)=0$.
With these observations,
\begin{align*}
	\lim_{r \searrow 0} \bP_r((\mathcal{A}_{r,1} \cap \mathcal{A}_{r,2})^c)
	= \lim_{r \searrow 0} \bigl\{ \bP_r(\mathcal{A}_{r,1}^c)+\bP_r(\mathcal{A}_{r,1} \cap \mathcal{A}_{r,2}^c) \bigr\}
	=0,
\end{align*}
and \eqref{eq:activating} is proved.
\end{proof}

\subsubsection{Recursion}\label{subsect:recursion}
Based on the active frogs cultivated in the initial box $\Lambda_r$, we recursively propagate active frogs at an appropriate speed.
To do this, set $\delta:=(50d)^{-d/2}$ and recall that $\Cr{CKN}=\Cr{CKN}(d,\delta) \in (0,1)$ is the constant appearing in Lemma~\ref{lem:CKN_cor} (with the $\delta$ above).
Furthermore, define for any $r \in (0,1]$,
\begin{align*}
	\psi_d(r):=2d\Cr{CKN}^{-1}\delta_d(r)^2,\qquad \nu_d(r):=100d\lceil \psi_d(r) \rceil,
\end{align*}
where $\delta_d(\cdot)$ is the function given by \eqref{eq:delta}.
For each $\xi \sim 0$, we consider the boxes
\begin{align*}
	Q_{r,i}(\xi):=B_\infty\bigl( 3\lceil \psi_d(r)^{1/2} \rceil i\xi,\lceil \psi_d(r)^{1/2} \rceil \bigr) \cap \Z^d,\qquad i \in \N_0.
\end{align*}
Then, the sets $\Gamma_{r,i}(\xi)$, $i \in \N_0$, are inductively constructed as follows:
\begin{align*}
	&\Gamma_{r,0}(\xi):=Q_{r,0}(\xi) \cap \mathcal{O},\\
	&\Gamma_{r,i}(\xi):=\mathcal{R}_{\nu_d(r)}^{\Gamma_{r,i-1}(\xi)} \cap Q_{r,i}(\xi) \cap \mathcal{O},\qquad i \in \N.
\end{align*}
In addition, $\sigma_r(\xi)$ stands for the first index $i$ such that the active frogs starting from $\Gamma_{r,i-1}(\xi)$ fail to meet at least $\lceil 2d\Cr{CKN}^{-1}|\log r| \rceil$ sleeping frogs in $Q_{r,i}(\xi)$ by time $\nu_d(r)$, i.e., 
\begin{align*}
	\sigma_r(\xi):=\inf\bigl\{ i \in \N_0:\#\Gamma_{r,i}(\xi)<2d\Cr{CKN}^{-1}|\log r| \bigr\}.
\end{align*}

The following lemma states the number of consecutive successful transmissions of active frogs between adjacent boxes.

\begin{lem}\label{lem:connecting}
We have for any $\xi \sim 0$,
\begin{align*}
	\lim_{r \searrow 0} \bP_r\Bigl( \sigma_r(\xi) \leq r^{-d/2}\psi_d(r)^{-1/2} \Bigr)=0.
\end{align*}
\end{lem}
\begin{proof}
Fix $\xi \sim 0$ and let $r \in (0,1]$ be small enough to justify the argument below.
The definition of $\sigma_r(\xi)$ implies that
\begin{align}\label{eq:sigma}
\begin{split}
	&\bP_r\Bigl( \sigma_r(\xi) \leq r^{-d/2}\psi_d(r)^{-1/2} \Bigr)\\
	&= \P_r\Bigl( \#\Gamma_{r,0}(\xi)<2d\Cr{CKN}^{-1}|\log r| \Bigr)\\
	&\quad +\sum_{i=1}^{\lfloor r^{-d/2}\psi_d(r)^{-1/2} \rfloor} \bP_r\Bigl( \#\Gamma_{r,0}(\xi) \geq 2d\Cr{CKN}^{-1}|\log r|,\,\sigma_r(\xi)=i \Bigr).
\end{split}
\end{align}
Use the Chernoff bound for a sum of independent Bernoulli random variables to estimate the first term of the right side in \eqref{eq:sigma}:
\begin{align*}
	\P_r\Bigl( \#\Gamma_{r,0}(\xi)<2d\Cr{CKN}^{-1}|\log r| \Bigr)
	&\leq \P_r\Bigl( \#\Gamma_{r,0}(\xi)<\frac{r}{2}\# Q_{r,0}(\xi) \Bigr)\\
	&\leq \exp\Bigl\{ -\frac{1}{8}(2d\Cr{CKN}^{-1}|\log r|) \Bigr\}
	\leq r^{d/4}.
\end{align*}
Here we used the fact that $\Cr{CKN} \in (0,1)$ in the last inequality.
Let us next estimate the second term of the right side in \eqref{eq:sigma}.
Since $\# Q_{r,i}(\xi) \geq \delta \nu_d(r)^{d/2}$ holds (recall $\delta=(50d)^{-d/2}$), Lemma~\ref{lem:CKN_cor} implies that for any finite subset $\Gamma$ of $\Z^d$ with $\max\{ \|x-y\|_2:x \in \Gamma,\, y \in Q_{r,i}(\xi) \} \leq \sqrt{\nu_d(r)}$,
\begin{align}
\begin{split}\label{eq:CKN_app}
	&\bP_r\biggl( \#\bigl( \mathcal{R}_{\nu_d(r)}^\Gamma \cap Q_{r,i}(\xi) \cap \mathcal{O} \bigr)<\frac{r}{2}\min\bigl\{ \Cr{CKN}\phi_d(\nu_d(r))\,\#\Gamma,(1-\delta)\,\# Q_{r,i}(\xi) \bigr\} \biggr)\\
	&\leq \exp\{ -\Cr{CKN}\,\#\Gamma \}+\exp\Bigl\{ -\frac{r}{8}\min\bigl\{ \Cr{CKN}\phi_d(\nu_d(r))\,\#\Gamma,(1-\delta)\,\# Q_{r,i}(\xi) \bigr\} \Bigr\}.
\end{split}
\end{align}
By the definitions of $\phi_d(\cdot)$, $\nu_d(\cdot)$ and $Q_{r,i}(\xi)$, one has
\begin{align*}
	\phi_d(\nu_d(r)) \geq 4d\Cr{CKN}^{-1}r^{-1},\qquad
	\# Q_{r,i}(\xi) \geq \frac{4d\Cr{CKN}^{-1}}{1-\delta}r^{-1}|\log r|,
\end{align*}
which leads to
\begin{align*}
	\min\bigl\{ \Cr{CKN}\phi_d(\nu_d(r))\,\#\Gamma,(1-\delta)\,\# Q_{r,i}(\xi) \bigr\}
	\geq 4dr^{-1} \min\{ \#\Gamma,\Cr{CKN}^{-1}|\log r| \}.
\end{align*}
This together with \eqref{eq:CKN_app} tells us that if $\#\Gamma \geq 2d\Cr{CKN}^{-1}|\log r|$ holds, then
\begin{align*}
	&\bP_r\Bigl( \#\bigl( \mathcal{R}_{\nu_d(r)}^\Gamma \cap Q_{r,i}(\xi) \cap \mathcal{O} \bigr)<2d\Cr{CKN}^{-1}|\log r| \Bigr)\\
	&\leq \exp\bigl\{ -\Cr{CKN}(2d\Cr{CKN}^{-1}|\log r|) \bigr\}+\exp\Bigl\{ -\frac{r}{8}(4d\Cr{CKN}^{-1}r^{-1}|\log r|) \Bigr\}
	\leq 2r^{d/2}.
\end{align*}
Note that for any $i \in \N$, $\Gamma_{r,i-1}(\xi)$ is independent of $(S_\cdot^z)_{z \in Q_{r,i-1}(\xi)}$ under $P$ and $\max\{ \|x-y\|_2:x \in Q_{r,i-1}(\xi),\,y \in Q_{r,i}(\xi) \} \leq \sqrt{\nu_d(r)}$ holds.
Hence, for any $i \in \N$,
\begin{align*}
	&\bP_r\Bigl( \#\Gamma_{r,0}(\xi) \geq 2d\Cr{CKN}^{-1}|\log r|,\,\sigma_r(\xi)=i \Bigr)\\
	&\leq \sum_\Gamma \E_r\Bigl[ P\bigl( \Gamma_{r,i-1}(\xi)=\Gamma \bigr) P\Bigl(\,\#\bigl( \mathcal{R}_{\nu_d(r)}^\Gamma \cap Q_{r,i}(\xi) \cap \mathcal{O} \bigr)<2d\Cr{CKN}^{-1}|\log r| \Bigr) \Bigr]\\
	&\leq 2r^{d/2},
\end{align*}
where the sum is taken over all subsets $\Gamma$ of $Q_{r,i-1}(\xi)$ with $\#\Gamma \geq 2d\Cr{CKN}^{-1}|\log r|$.
Since $r^{-d/2}\psi_d(r)^{-1/2} \leq r^{-(d-1)/2}$, it follows that
\begin{align*}
	\sum_{i=1}^{\lfloor r^{-d/2}\psi_d(r)^{-1/2} \rfloor} \bP_r\Bigl( \#\Gamma_{r,0}(\xi) \geq 2d\Cr{CKN}^{-1}|\log r|,\,\sigma_r(\xi)=i \Bigr)
	\leq 2\sqrt{r}.
\end{align*}
Consequently, for all $r$ sufficiently small (depending on $d$),
\begin{align*}
	\bP_r\Bigl( \sigma_r(\xi) \leq r^{-d/2}\psi_d(r)^{-1/2} \Bigr)
	\leq r^{d/4}+2\sqrt{r}
	\leq 3\sqrt{r},
\end{align*}
and the lemma follows by letting $r \searrow 0$.
\end{proof}

We now choose
\begin{align}\label{eq:rho}
	\rho=\rho(d):=410d^2\Cr{CKN}^{-1}
\end{align}
(which is the constant appearing in the concept of $r$-good, see Section~\ref{subsect:pf_upper}), and prove Proposition~\ref{prop:good}.

\begin{proof}[\bf Proof of Proposition~\ref{prop:good}]
By the translation invariance of $(\omega,S)$, it suffices to show that
\begin{align}\label{eq:good_goal}
	\lim_{r \searrow 0} \bP_r(\text{$0$ is $r$-good})=1.
\end{align}
It is clear that for any $\xi \sim 0$ and $i \in \N_0$, the side length of $Q_{r,i}(\xi)$ is bounded from above by a multiple of $\sqrt{r^{-1}|\log r|}$, and so is the $\ell^1$-distance between the centers of $Q_{r,i}(\xi)$ and $Q_{r,i+1}(\xi)$.
This implies that for all $r$ sufficiently small, $Q_{r,0}(\xi) \subset \Lambda_r$ and $Q_{r,i}(\xi) \subset B_\infty(\lceil r^{-d/2} \rceil \xi,r^{-(1/2+\epsilon)})$ for some $i \in [0,r^{-d/2}\psi_d(r)^{-1/2}]$.
Hence, if the event
\begin{align*}
	\mathcal{A}_r \cap \bigcap_{\xi \sim 0} \bigl\{ \sigma_r(\xi)>r^{-d/2}\psi_d(r)^{-1/2} \bigr\}
\end{align*}
happens, then $\Theta_r^\text{in}(0)=B_\infty(0,r^{-(1/2+\epsilon)})$ contains at least one occupied site, and for any occupied site $x \in \Theta_r^\text{in}(0)$ and $\xi \sim 0$, there exists an occupied site $y \in B_\infty(\lceil r^{-d/2} \rceil \xi,r^{-(1/2+\epsilon)})$ such that
\begin{align*}
	T_{B_\infty(0,2\lceil r^{-d/2} \rceil)}(x,y)
	&\leq \max_{z \in \Lambda_r} T_{\Lambda_r}(x,z)+r^{-d/2}\psi_d(r)^{-1/2} \times \nu_d(r)\\
	&\leq 5dr^{-(d+1)/2}+200dr^{-d/2}\psi_d(r)^{1/2}\\
	&\leq \rho r^{-d/2}\delta_d(r),
\end{align*}
which implies that $0$ is $r$-good.
Therefore, the union bound, Proposition~\ref{prop:activating} and Lemma~\ref{lem:connecting} show that
\begin{align*}
	&\lim_{r \searrow 0} \bP_r(\text{$0$ is not $r$-good})\\
	&\leq \lim_{r \searrow 0} \biggl\{ \bP_r(\mathcal{A}_r^c)+\sum_{\xi \sim 0} \bP_r\Bigl( \sigma_r(\xi) \leq r^{-d/2}\psi_d(r)^{-1/2} \Bigr) \biggr\}
	=0,
\end{align*}
and \eqref{eq:good_goal} follows.
\end{proof}

\appendix
\section{Proof of Proposition~\ref{prop:CKN}}\label{app:CKN}
This appendix is devoted to the proof of Proposition~\ref{prop:CKN}.
It was originally proved in \cite{CanKubNak25_IP}, but we reproduce it for the convenience of the reader.

Let us first estimate how many sites in any relatively large subset of $B_2(0,\sqrt{n}) \cap \Z^d$ the simple random walk starting at $0$ can visit by time $n$.

\begin{lem}\label{lem:range_lower}
For any $\epsilon>0$, there exists a constant $\Cl{app1}=\Cr{app1}(d,\epsilon) \in (0,1)$ such that for any $n \geq 2$ and for any subset $\Gamma$ of $B_2(0,\sqrt{n}) \cap \Z^d$ with $\#\Gamma \geq \epsilon n^{d/2}$,
\begin{align*}
	P\bigl( \#(\mathcal{R}_n^0 \cap \Gamma) \geq \Cr{app1}\phi_d(n) \bigr) \geq \Cr{app1},
\end{align*}
where $\phi_d(\cdot)$ is the function appearing in Proposition~\ref{prop:CKN}.
\end{lem}
\begin{proof}
We begin by proving that for any $\Gamma \subset \Z^d$ and $n \in \N$,
\begin{align}\label{eq:PZ_goal}
	P\biggl( \#(\mathcal{R}_n^0 \cap \Gamma) \geq \half E[\#(\mathcal{R}_n^0 \cap \Gamma)] \biggr)
	\geq \frac{E[\#(\mathcal{R}_n^0 \cap \Gamma)]}{12 \sup_{x \in \Gamma}E[\#(\mathcal{R}_n^x \cap \Gamma)]}.
\end{align}
The Paley--Zygmund inequality implies that for any $\Gamma \subset \Z^d$ and $n \in \N$,
\begin{align}\label{eq:PZ}
	P\biggl( \#(\mathcal{R}_n^0 \cap \Gamma) \geq \half E[\#(\mathcal{R}_n^0 \cap \Gamma)] \biggr)
	\geq \frac{E[\#(\mathcal{R}_n^0 \cap \Gamma)]^2}{4E[\{ \#(\mathcal{R}_n^0 \cap \Gamma) \}^2]}.
\end{align}
Note that the expectation in the denominator can be rewritten as follows:
\begin{align*}
	E\bigl[ \{\#(\mathcal{R}_n^0 \cap \Gamma)\}^2 \bigr]
	&= \sum_{x,y \in \Gamma} P(H(0,x) \leq n,H(0,y) \leq n)\\
	&= E[\#(\mathcal{R}_n^0 \cap \Gamma)]+\sum_{\substack{x,y \in \Gamma\\ x \not= y}} P(H(0,x) \leq n,H(0,y) \leq n).
\end{align*}
We use the strong Markov property to estimate the probabilities in the rightmost side from above: for any distinct $x,y \in \Gamma$,
\begin{align*}
	&P(H(0,x) \leq n,H(0,y) \leq n)\\
	&= P(H(0,x)<H(0,y) \leq n)+P(H(0,y)<H(0,x) \leq n)\\
	&\leq P(H(0,x) \leq n)P(H(x,y) \leq n)+P(H(0,y) \leq n)P(H(y,x) \leq n).
\end{align*}
Hence, from the fact that $E[\#(\mathcal{R}_n^x \cap \Gamma)] \geq 1$ for all $x \in \Gamma$,
\begin{align*}
	E\bigl[ \{ \#(\mathcal{R}_n^0 \cap \Gamma) \}^2 \bigr]
	&\leq E[\#(\mathcal{R}_n^0 \cap \Gamma)]+2\sum_{\substack{x,y \in \Gamma\\ x \not= y}} P(H(0,x) \leq n)P(H(x,y) \leq n)\\
	&\leq E[\#(\mathcal{R}_n^0 \cap \Gamma)]+2\sup_{x \in \Gamma} E[\#(\mathcal{R}_n^x \cap \Gamma)] \times E[\#(\mathcal{R}_n^0 \cap \Gamma)]\\
	&\leq 3E[\#(\mathcal{R}_n^0 \cap \Gamma)] \times \sup_{x \in \Gamma} E[\#(\mathcal{R}_n^x \cap \Gamma)].
\end{align*}
This combined with \eqref{eq:PZ} yields that
\begin{align*}
	P\biggl( \#(\mathcal{R}_n^0 \cap \Gamma) \geq \half E[\#(\mathcal{R}_n^0 \cap \Gamma)] \biggr)
	&\geq \frac{E[\#(\mathcal{R}_n^0 \cap \Gamma)]^2}{12E[\#(\mathcal{R}_n^0 \cap \Gamma)] \times \sup_{x \in \Gamma} E[\#(\mathcal{R}_n^x \cap \Gamma)]}\\
	&= \frac{E[\#(\mathcal{R}_n^0 \cap \Gamma)]}{12\sup_{x \in \Gamma} E[\#(\mathcal{R}_n^x \cap \Gamma)]},
\end{align*}
which proves \eqref{eq:PZ_goal}.

Let us complete the proof of the lemma.
Fix $n \in \N$ and a subset $\Gamma$ of $B_2(0,\sqrt{n}) \cap \Z^d$ with $\#\Gamma \geq \epsilon n^{d/2}$.
Then, by \eqref{eq:AMP}, there exists a constant $c \in (0,1)$ (which depends only on $d$) such that for all $n \in \N$,
\begin{align*}
	E[\#(\mathcal{R}_n^0 \cap \Gamma)]
	&= \sum_{x \in \Gamma \setminus \{0\}} P(H(0,x) \leq n)+\1{\{ 0 \in \Gamma \}}\\
	&\geq c(\#\Gamma) \times
	\begin{dcases*}
		\frac{1}{\log(1+\sqrt{n})}, & if $d=2$,\\[0.5em]
		n^{1-d/2}, & if $d \geq 3$.
	\end{dcases*}
\end{align*}
Since $\#\Gamma \geq \epsilon n^{d/2}$, this implies that $E[\#(\mathcal{R}_n^0 \cap \Gamma)] \geq (2c/3)\epsilon \phi_d(n)$ for all $n \geq 2$.
On the other hand, \cite[Theorem~1]{DvoErd51} guarantees that there exists a constant $c'$ (which depends only on $d$) such that $\sup_{x \in \Gamma}E[\#(\mathcal{R}^x_n \cap \Gamma)] \leq E[\#\mathcal{R}^0_n] \leq c'\phi_d(n)$ for all $n \in \N$.
Therefore, \eqref{eq:PZ_goal} shows that for all $n \geq 2$, 
\begin{align*}
	P\biggl( \#(\mathcal{R}_n^0 \cap \Gamma) \geq \frac{1}{3} \epsilon c\mathcal{\phi}_d(n) \biggr)
	&\geq P\biggl( \#(\mathcal{R}_n^0 \cap \Gamma) \geq \half E[\#(\mathcal{R}_n^0 \cap \Gamma)] \biggr)\\
	&\geq \frac{E[\#(\mathcal{R}_n^0 \cap \Gamma)]}{12 \sup_{x \in \Gamma}E[\#(\mathcal{R}_n^x \cap \Gamma)]}\\
	&\geq \frac{c\epsilon \phi_d(n)}{18c'\phi_d(n)}
	= \frac{c\epsilon}{18c'},
\end{align*}
and the desired lower bound is obtained by taking $\Cr{app1}:=c\epsilon/(18c')$.
\end{proof}

We will also use the following upper deviation bound for a sum of Bernoulli random variables. A proof can be found for example in~\cite[Corollary~2.4.7]{DemZei10_book}. 

\begin{lem}\label{lem:chernoff}
Let $q \in (0,1)$ and let $(\mathcal{F}_i)_{i=0}^\infty$ be a filtration with the trivial $\sigma$-field $\mathcal{F}_0$.
Moreover, assume that $(X_i)_{i=1}^\infty$ is a family of $(\mathcal{F}_i)_{i=1}^\infty$-adapted $\{0,1\}$-valued random variables satisfying that for any $i=1,\dots,\ell$, $Q(X_i=1|\mathcal{F}_{i-1}) \leq q$ holds almost surely (where $Q$ is the underlying probability measure of $(X_i)_{i=1}^\infty$).
Then, there exists a constant $\Cl{app2} \in (0,1)$ (which depends only on $q$) such that for all $n \in \N$,
\begin{align*}
	Q\biggl( \sum_{i=1}^n X_i \geq (1-\Cr{app2})n \biggr) \leq \exp\{ -\Cr{app2}n \}.
\end{align*}
\end{lem}

We are now in a position to prove Proposition~\ref{prop:CKN}.

\begin{proof}[\bf Proof of Proposition~\ref{prop:CKN}]
Fix $\delta \in (0,1)$ and $n \geq 2$.
Let $A$ and $B$ be finite subsets of $\Z^d$ with $\max\{ \| x-y \|_2:x \in A,\,y \in B \} \leq \sqrt{n}$ and $\# B \geq \delta n^{d/2}$.
We now enumerate $A=\{ x_1,\dots,x_\ell \}$ (with $\ell:=\# A$), and define $\Gamma_i:=B \setminus \bigcup_{j=1}^{i-1}\mathcal{R}_n^{x_j}$ for each $i=1,\dots,\ell$.
Moreover, let $\tau$ denote the smallest integer $i \in [1,\ell]$ such that $\#\Gamma_i \leq \delta\,\# B$, with the convention that $\tau:=\infty$ if no such integer exists:
\begin{align*}
	\tau:=\inf\{ i \in [1,\ell]:\#\Gamma_i \leq \delta\,\# B \}.
\end{align*}
Note that if $\tau<\ell$ holds, then
\begin{align*}
	\#\bigl( \mathcal{R}_n^A \cap B \bigr)
	\geq \#\Biggl( \bigcup_{i=1}^\tau \mathcal{R}_n^{x_i} \cap B \Biggr)
	\geq (1-\delta)\,\# B.
\end{align*}
This implies that for any $c \in (0,1)$,
\begin{align*}
	&P\Bigl( \#\bigl( \mathcal{R}_n^A \cap B \bigr)<\min\{ c\phi_d(n)\,\# A,(1-\delta)\,\# B \} \Bigr)\\
	&\leq P\Bigl( \tau<\ell,\,\#\bigl( \mathcal{R}_n^A \cap B \bigr)<(1-\delta)\,\# B \Bigr)+P\Bigl( \tau \geq \ell,\,\#\bigl( \mathcal{R}_n^A \cap B \bigr)<c\phi_d(n) \ell \Bigr)\\
	&= P\Bigl( \tau \geq \ell,\,\#\bigl( \mathcal{R}_n^A \cap B \bigr)<c\phi_d(n) \ell \Bigr).
\end{align*}
Therefore, once we can find a constant $c \in (0,1)$ (which depends only on $d$ and $\delta$) such that
\begin{align}\label{eq:app_goal}
	P\Bigl( \tau \geq \ell,\,\#\bigl( \mathcal{R}_n^A \cap B \bigr)<c\phi_d(n) \ell \Bigr)
	\leq \exp\{-c \ell \},
\end{align}
the proposition follows since $\ell=\# A$.

To find a constant $c=c(d,\delta) \in (0,1)$ satisfying \eqref{eq:app_goal}, for each $i=1,\dots,\ell$, we introduce the $\sigma$-field $\mathcal{F}_i$ generated by $S_\cdot^{x_1},\dots,S_\cdot^{x_i}$ and define
\begin{align*}
	Y_i:=
	\begin{cases*}
		1, & if $\#(\mathcal{R}_n^{x_i} \cap \Gamma_i)<\Cr{app1}\phi_d(n)$,\\
		0, & otherwise.
	\end{cases*}	
\end{align*}
Note that $(\mathcal{F}_i)_{i=0}^\ell$ is a filtration (with the convention that $\mathcal{F}_0$ is the trivial $\sigma$-field) and $(Y_i\1{\{ \tau \geq i \}})_{i=1}^\ell$ is a family of $(\mathcal{F}_i)_{i=0}^\ell$-adapted Bernoulli random variables.
Since, for each $i=1,\dots,\ell$, both the random set $\Gamma_i$ and the event $\{ \tau \geq i \}$ are $\mathcal{F}_{i-1}$-measurable, Lemma~\ref{lem:range_lower} with $\epsilon=\delta^2$ implies that that for any $i=1,\dots,\ell$, $P$-almost surely,
\begin{align*}
	P\bigl( Y_i\1{\{ \tau \geq i \}}=1 \big| \mathcal{F}_{i-1} \bigr)
	&= \sum_{\substack{\Gamma \subset B\\\#\Gamma \geq \delta\,\# B}}P\bigl( \#(\mathcal{R}_n^{x_i} \cap \Gamma)<\Cr{app1}\phi_d(n) |\mathcal{F}_{i-1} \bigr)\1{\{ \Gamma_i=\Gamma,\,\tau \geq i \}}\\
	&= \sum_{\substack{\Gamma \subset B\\\#\Gamma \geq \delta\,\# B}}P\bigl( \#(\mathcal{R}_n^{x_i} \cap \Gamma)<\Cr{app1}\phi_d(n) \bigr)\1{\{ \Gamma_i=\Gamma,\,\tau \geq i \}}\\
	&\leq 1-\Cr{app1}.
\end{align*}
Here, in the second equality, we used the fact that $\mathcal{R}_n^{x_i}$ is independent of $\mathcal{F}_{i-1}$.
Hence, we can use Lemma~\ref{lem:chernoff} with $X_i=Y_i\1{\{ \tau \geq i \}}$ and $q=1-\Cr{app1}$ to obtain
\begin{align*}
	P\biggl( \sum_{i=1}^\ell Y_i\1{\{ \tau \geq i \}} \geq (1-\Cr{app2})\ell \biggr) \leq \exp\{ -\Cr{app2}\ell \}.
\end{align*}
It is clear that if $\sum_{i=1}^\ell Y_i<(1-\Cr{app2})\ell$ (or equivalently $\sum_{i=1}^\ell (1-Y_i)>\Cr{app2}\ell$), then
\begin{align*}
	\#\bigl( \mathcal{R}_n^A \cap B \bigr)
	= \sum_{i=1}^\ell \#\bigl( \mathcal{R}_n^{x_i} \cap \Gamma_i \bigr)
	\geq \Cr{app1}\Cr{app2}\phi_d(n)\ell.
\end{align*}
In addition, for any $i=1,\dots,\ell$, $\1{\{ \tau \geq i \}}=1$ holds on the event $\{ \tau \geq \ell \}$.
Therefore, by taking $\Cr{CKN}:=\Cr{app1}\Cr{app2}$, one has
\begin{align*}
	P\Bigl( \tau \geq \ell,\,\#\bigl( \mathcal{R}_n^A \cap B \bigr)<\Cr{CKN}\phi_d(n)\ell \Bigr)
	&\leq P\biggl(\tau \geq \ell,\,\sum_{i=1}^\ell Y_i\1{\{ \tau \geq i \}} \geq (1-\Cr{app2})\ell \biggr)\\
	&\leq \exp\{ -\Cr{app2}\ell \} \leq \exp\{ -\Cr{CKN}\ell \},
\end{align*}
and $\Cr{CKN}$ is the desired constant satisfying \eqref{eq:app_goal}.
\end{proof}

\section*{Acknowledgements}
The authors thank Can Van Hao and Shuta Nakajima for permitting the inclusion of a result in~\cite{CanKubNak25_IP} in this paper.
R.F.~was supported by JSPS KAKENHI Grant Number JP22H00099 and JP25K00911.
N.K.~was supported by JSPS KAKENHI Grant Number JP20K14332 and JP25K07054.


\begin{thebibliography}{10}

\bibitem{AlvMacPop02}
O.~S.~M. Alves, F.~P. Machado, and S.~Y. Popov.
\newblock The shape theorem for the frog model.
\newblock {\em The Annals of Applied Probability}, 12(2):533--546, 2002.

\bibitem{AlvMacPopRav01}
O.~S.~M. Alves, F.~P. Machado, S.~Y. Popov, and K.~Ravishankar.
\newblock The shape theorem for the frog model with random initial
  configuration.
\newblock {\em Markov Process. Related Fields}, 7(4):525--539, 2001.

\bibitem{BecDinDurHuoJun18}
E.~Beckman, E.~Dinan, R.~Durrett, R.~Huo, and M.~Junge.
\newblock Asymptotic behavior of the {B}rownian frog model.
\newblock {\em Electronic Journal of Probability}, 23, 2018.

\bibitem{BerRam16}
J.~B{\'e}rard and A.~Ram{\'\i}rez.
\newblock Fluctuations of the front in a one-dimensional model for the spread
  of an infection.
\newblock {\em The Annals of Probability}, 44(4):2770--2816, 2016.

\bibitem{BerRam10}
J.~B{\'e}rard and A.~F. Ram{\'\i}rez.
\newblock Large deviations of the front in a one-dimensional model of {$X+Y \to
  2X$}.
\newblock {\em The Annals of Probability}, 38(3):955--1018, 2010.

\bibitem{CanKubNak25_IP}
V.~H. Can, N.~Kubota, and S.~Nakajima.
\newblock Upper tail large deviation for the first passage time in the frog
  model on $\mathbb{Z}^d$ with $d \geq 2$.
\newblock in preparation.

\bibitem{CanKubNak23_arXiv}
V.~H. Can, N.~Kubota, and S.~Nakajima.
\newblock Upper tail large deviation for the one-dimensional frog model.
\newblock {\em arXiv:2312.02745}, 2023.

\bibitem{CanKubNak25}
V.~H. Can, N.~Kubota, and S.~Nakajima.
\newblock Lipschitz-type estimate for the frog model with {B}ernoulli initial
  configuration.
\newblock {\em Mathematical Physics, Analysis and Geometry}, 28(1):1--40, 2025.

\bibitem{CanNak19}
V.~H. Can and S.~Nakajima.
\newblock First passage time of the frog model has a sublinear variance.
\newblock {\em Electron. J. Probab.}, 24:Paper No. 76, 27, 2019.

\bibitem{DemZei10_book}
A.~Dembo and O.~Zeitouni.
\newblock Large deviations techniques and applications, volume 38 of stochastic
  modelling and applied probability, 2010.

\bibitem{DvoErd51}
A.~Dvoretzky and P.~Erd{\"o}s.
\newblock Some problems on random walk in space.
\newblock In {\em Proc. 2nd Berkeley Symp}, pages 353--367, 1951.

\bibitem{GarMar10}
O.~Garet and R.~Marchand.
\newblock Moderate deviations for the chemical distance in {B}ernoulli
  percolation.
\newblock {\em Alea}, 7:171--191, 2010.

\bibitem{Gri99_book}
G.~Grimmett.
\newblock {\em Percolation}, volume 321 of {\em Grundlehren der mathematischen
  Wissenschaften [Fundamental Principles of Mathematical Sciences]}.
\newblock Springer-Verlag, Berlin, second edition, 1999.

\bibitem{GuoTanWei22}
C.~Guo, S.~Tang, and N.~Wei.
\newblock On the minimal drift for recurrence in the frog model on d-ary trees.
\newblock {\em The Annals of Applied Probability}, 32(4):3004--3026, 2022.

\bibitem{HofJohJun17}
C.~Hoffman, T.~Johnson, and M.~Junge.
\newblock Recurrence and transience for the frog model on trees.
\newblock {\em The Annals of Probability}, 45(5):2826--2854, 2017.

\bibitem{JohJun18}
T.~Johnson and M.~Junge.
\newblock Stochastic orders and the frog model.
\newblock {\em Ann. Inst. Henri Poincar\'e{} Probab. Stat.}, 54(2):1013--1030,
  2018.

\bibitem{JohRol19}
T.~Johnson and L.~T. Rolla.
\newblock Sensitivity of the frog model to initial conditions.
\newblock {\em Electronic Communications in Probability}, 24:1--9, 2019.

\bibitem{KosZer17}
E.~Kosygina and M.~P. Zerner.
\newblock A zero-one law for recurrence and transience of frog processes.
\newblock {\em Probability Theory and Related Fields}, 168(1-2):317--346, 2017.

\bibitem{Kub19}
N.~Kubota.
\newblock Deviation bounds for the first passage time in the frog model.
\newblock {\em Advances in Applied Probability}, 51(1):184--208, 2019.

\bibitem{Kub20}
N.~Kubota.
\newblock Continuity for the asymptotic shape in the frog model with random
  initial configurations.
\newblock {\em Stochastic Processes and their Applications}, 130(9):5709 --
  5734, 2020.

\bibitem{LawLim10_book}
G.~F. Lawler and V.~Limic.
\newblock {\em Random walk: a modern introduction}, volume 123.
\newblock Cambridge University Press, 2010.

\bibitem{MicRos20}
M.~Michelen and J.~Rosenberg.
\newblock The frog model on non-amenable trees.
\newblock {\em Electron. J. Probab}, 25(49):1--16, 2020.

\bibitem{MulWie20}
S.~M{\"u}ller and G.~M. Wiegel.
\newblock On transience of frogs on galton--watson trees.
\newblock {\em Electronic Journal of Probability}, 25, 2020.

\bibitem{RamSid04}
A.~F. Ram\'{\i}rez and V.~Sidoravicius.
\newblock Asymptotic behavior of a stochastic combustion growth process.
\newblock {\em J. Eur. Math. Soc.(JEMS)}, 6(3):293--334, 2004.

\bibitem{TelWor99}
A.~Telcs and N.~C. Wormald.
\newblock Branching and tree indexed random walks on fractals.
\newblock {\em Journal of applied probability}, pages 999--1011, 1999.

\end{thebibliography}


\end{document}